\documentclass[12pt,reqno]{amsart}
\usepackage[margin = 1.3 in]{geometry}
\usepackage{cite}
\usepackage{
  hyperref,
  amsmath,
  amssymb,
  tikz,
  amsthm,
  thmtools,
  microtype,
  stmaryrd,
  tikz-cd,
  mathrsfs,
  pgfplots
}
\usepackage{graphicx}
\usepackage{paralist}
\usepackage[shortlabels]{enumitem}
\usepackage{import}
\usepackage{xifthen}
\usepackage{pdfpages}
\usepackage{transparent}
\usepackage{booktabs}
\usepackage{subcaption}

\usepackage[all]{xy}
\usepackage{eucal}

\title{A universal formula for counting cubic surfaces}
\author{Anand Deopurkar, Anand Patel
  \& Dennis Tseng}

\dedicatory{Dedicated to our teacher, Joe Harris, on the occasion of
  his $70$th birthday.}

\renewcommand{\k}{k}
\DeclareMathOperator{\id}{id}
\DeclareMathOperator{\Bl}{Bl}
\DeclareMathOperator{\Orb}{\overline{Orb}}
\DeclareMathOperator{\Polar}{Polar}

\DeclareMathOperator{\Pole}{Pole}
\DeclareMathOperator{\Hilb}{Hilb}

\DeclareMathOperator{\M}{\mathcal{M}}
\DeclareMathOperator{\mult}{mult}
\renewcommand{\to}{{\longrightarrow}}

\ifcsname theorem\endcsname{}\else\declaretheorem[parent=subsection]{theorem}\fi
\ifcsname corollary\endcsname{}\else\declaretheorem[sibling=theorem]{corollary}\fi
\ifcsname lemma\endcsname{}\else\declaretheorem[sibling=theorem]{lemma}\fi
\ifcsname proposition\endcsname{}\else\declaretheorem[sibling=theorem]{proposition}\fi
\ifcsname conjecture\endcsname{}\else\fi
\ifcsname problem\endcsname{}\else\fi
\ifcsname question\endcsname{}\else\fi
\ifcsname definition\endcsname{}\else\declaretheorem[sibling=theorem, style=definition]{definition}\fi
\ifcsname exercise\endcsname{}\else\fi
\ifcsname example\endcsname{}\else\declaretheorem[sibling=theorem, style=definition]{example}\fi
\ifcsname remark\endcsname{}\declaretheorem[sibling=theorem, style=remark]{remark}\fi

\providecommand {\Z}{{\bf Z}}
\providecommand {\Q}{{\bf Q}}

\renewcommand {\P}{{\bf P}}

\providecommand {\A}{{\bf A}}

\providecommand{\GL}{\operatorname{GL}}
\providecommand{\PGL}{\operatorname{PGL}}

\providecommand{\spec}{\operatorname{Spec}}

\providecommand{\Aut}{\operatorname{Aut}}

\providecommand{\Pic}{\operatorname{Pic}}
\providecommand{\Sym}{\operatorname{Sym}}
\providecommand{\rk}{\operatorname{rk}}
\declaretheorem[sibling=theorem,style=remark]{remark}
\numberwithin{equation}{section}

\renewcommand{\O}{\mathcal O}
\newcommand{\G}{\mathbf G}
\newcommand{\F}{\mathbf F}
\newcommand{\td}{\widetilde}
\newcommand{\V}{\mathcal V}

\newcommand{\cX}{\mathcal{X}}

\newcommand{\smvee}{\raise0.5ex\hbox{$\scriptscriptstyle\vee$}}

\newcommand{\Proj}{{\text{\rm Proj}\,}}

\begin{document}

\maketitle

\begin{abstract}
  Using equivariant geometry, we find a universal formula that computes the number of times a general cubic surface arises in a family.
  As applications, we show that the $\PGL_4$ orbit closure of a generic cubic surface has degree 96120, and that a general cubic surface arises 42120 times as a hyperplane section of a general cubic 3-fold.
\end{abstract}

\section{Introduction}
\label{sec:intro}

Cubic surfaces have fascinated us for centuries. Ever since Cayley and
Salmon discovered the 27 lines, it has become somewhat of a tradition
to revisit their geometry, using increasingly sophisticated tools to
uncover more and more. Our contribution to this story provides a
universal answer to a collection of natural enumerative questions
about families of cubic surfaces.

The isomorphism classes of cubic surfaces form a 4-dimensional moduli
space. So, we expect that in a 4-dimensional family, a general
isomorphism class would appear finitely many times. How can we
determine this number? For example, how many times does a general
cubic surface appear, up to isomorphism, as
\begin{enumerate}[(A)]
\item \label{orbit} a member of a 4-dimensional general linear system of cubic surfaces in $\P^3$?
\item \label{slice} a hyperplane section of a general cubic threefold in $\P^4$?
\item a blow up of 6 points in $\P^2$ varying in a general
  4-dimensional linear system on a cubic curve?
\end{enumerate}
Our formula answers all questions of this kind.  The answers to these
specific questions are $96120$, $42120$, and $25920$, respectively.

To state our formula, we must fix a precise moduli space of cubics.
Let us denote by $\mathcal M$ the GIT quotient of the set of
semistable cubic surfaces in $\P^3$ by the action of $\PGL_4$.  By a
\emph{family of cubic surfaces}, we mean a flat, proper morphism whose
geometric fibers are isomorphic to cubic surfaces in $\P^3$.

\begin{theorem}\label{theorem:main}
  Let $\pi \colon \cX \to B$ be a good family of cubic surfaces over a
  proper base $B$.
  Let $\V$ denote the rank 4 vector bundle
  $ \pi_* \left(\omega_\pi^{-1}\right)$, and let
  \(v_i = c_i\left(\V\right)\) be its Chern classes.  Assume that a
  general fiber of $\pi$ is semistable, and let
  $\mu \colon B \dashrightarrow \mathcal M$ be the induced moduli map.  Then,
  \begin{align}
    \deg \mu =
    \label{eq:MAIN}
    1080 \int_{B} \left(v_{1}^{2}v_{2} - v_{1}v_{3}+ 9v_{4}\right).
  \end{align}
\end{theorem}
\autoref{def:goodfamily} gives the precise meaning of a good family.  This notion includes any family whose fibers are
GIT-semistable or have a finite automorphism group.

\subsection{Degree of the orbit closure}
Question~\ref{orbit} in \autoref{sec:intro} fits in a broader story
that has been an ongoing topic of study.  Consider a hypersurface
\(X\) in a projective space $\P V$ cut out by a homogeneous polynomial
$F \in \Sym^d V^\vee$.  In $\P\Sym^dV^\vee$, consider the orbit of
$[F]$ under the action of $\GL(V)$, and let $\Orb([F])$ be its
closure.  What is the degree of $\Orb([F])$?  To our knowledge,
Enriques and Fano in 1897 were the first to address this question.
They computed the degree for simple binary forms of all degrees
\cite{enr.fan:97}.  In a series of seminal papers \cite{Aluffi1993,
  alu.fab:93, alu.fab:00}, Aluffi and Faber computed the degree of the
orbit closure for \emph{all} binary and ternary forms, finishing off
the cases of points in $\P^1$ and curves in $\P^2$.  The next case is
that of cubic surfaces in $\P^3$.  \autoref{theorem:main} applied to
Question~\ref{orbit} settles it for general cubic surfaces---the
degree of the orbit closure of a general cubic surface is $96120$.

Question~\ref{orbit} appears as one of the ``twenty-seven questions
about the cubic surface'' by Ranestad and Sturmfels \cite{ran.stu:20}.
In 2019, Brustenga i Moncus\'i, Timme, and Weinstein gave a
``numerical proof'' for the number 96120. In 2020, Cazzador and
Skauli outlined an approach towards a proof by following the
techniques of Aluffi and Faber \cite{bru-i-mon.tim.wei:20,caz.ska:20}.
Our approach is different.

\subsection{The equivariant class of the orbit closure}
Our key idea is to consider not just the degree of the orbit closure,
but the class of the orbit closure in the equivariant Chow ring
$A^\bullet_{\GL V}(\Sym^dV^\vee)$.
The equivariant class subsumes the
degree, but, counter-intuitively, it is easier to find.

Let us explain how the equivariant class yields the degree.
We have natural maps
\[
\begin{tikzcd}
  A^\bullet_{\GL V}(\Sym^dV^\vee) \ar{r}& A^\bullet_{\GL V}(\Sym^d V^\vee \setminus
  \{0\}) \otimes \Q \ar[equal]{d}\\
  &A^\bullet_{\PGL V}(\P \Sym^dV^\vee) \otimes \Q \ar{r}&
  A^\bullet(\P \Sym^dV^\vee) \otimes \Q.
\end{tikzcd}
\]
Under the composite map, the class of $\Orb(F) \in A^\bullet_{\GL
  V}(\Sym^dV^\vee)$ is mapped to the class of $\Orb([F]) \in
A^\bullet(\P \Sym^dV^\vee) \otimes \Q$, which is simply the degree of $\Orb([F])$.

Let us explain what makes it possible to find the equivariant
class.  The ring $A^\bullet_{\GL V}(\Sym^dV^\vee)$ is the polynomial ring
in variables $c_1, \dots, c_r$, where $r = \dim V$, and the variable
$c_i$ has degree $i$.  The equivariant class of $\Orb(F)$ lies in the
graded component $A^m(\Sym^dV^\vee)$, where
$m = \dim \Sym^d V^\vee - \dim \GL V$.  As a result, it can be expressed as
a $\Z$-linear combination of monomials in $c_i$ of total degree
$m$. We can now use the method of undetermined coefficients.  We
construct maps $B \to [\Sym^dV^\vee/\GL V]$ where the pull-back of
$\Orb(F)$ as well as the classes $c_i$ can be explicitly
computed. Every such family gives a linear relationship between the
coefficients in the expression of $\Orb(F)$, which are to be determined.
If we have enough families, we get enough information to find all the coefficients.
For cubic surfaces, we construct more than enough families---there are 5
coefficients to be determined and we construct 8 families. The redundancy
serves as an extra check on the result.
\begin{theorem}\label{thm:eqvclass2}
  Let $V$ be a 4-dimensional vector space and let $F$ be a general
  element of $\Sym^3 V^\vee$.
  The equivariant class of the $\GL(V)$-orbit closure $\Orb(F)$ of $F$
  in $A^\bullet_{\GL V}(\Sym^3 V^\vee)$ is given by
  \[
    [\Orb F] = 24c_1^4 + 12c_1^2c_2 - 6c_1c_3 + 9c_4,
  \]
  where the \(c_i = c_i(V)\) are the Chern classes of the tautological vector
  bundle $V$.
\end{theorem}
The formula in \autoref{thm:eqvclass2} appears to be different from the one
in \autoref{theorem:main}, but they are related by a simple change of
variables; see \autoref{proof:eqvclass2}.

Most of the paper is devoted to constructing {\sl provably} good
families of cubic surfaces on which we can compute $\deg \mu$ and the
Chern classes $c_i$.  Proofs of goodness are sometimes challenging---
they require solving instances of the problem of determining which
homogeneous forms lie in the orbit closure of which other forms.
To give a teaser for our test families, we list their
bases $B$ here: $\P^4$, $\overline M_{0,7}$, a
Hassett-weighted variant of $\overline M_{0,7}$, a blow up of the
Hilbert scheme of two points on a quintic del Pezzo surface, and last
but not least, the stack $B\G_m$.

The method of undetermined coefficients applies to hypersurfaces of
any degree and dimension.  The obstacle, in general, is in
constructing enough families.  Our constructions crucially use the
beautiful but specific geometry cubic surfaces.

\subsection{The orbit closure problem}
To apply the method of undetermined coefficients, we have to ensure that no fibers of our family lie in the orbit closure of generic cubic surfaces (otherwise, we get excess intersection).
Thus, we have to prove statements of the form ``a given cubic surface $S$ is not in the closure of a suitably generic cubic surface.''
This is, in general, a difficult problem.
We develop a range of tools to address problems of this kind.
To preserve the flow of the paper, and because these techniques may be of independent interest, we assemble these tools in a separate section (\autoref{sec:orbclosure}).

\subsection{Further remarks}
The formula in \autoref{theorem:main} has some curious consequences.
\begin{corollary}
  \label{cor:ineq} If $\pi \colon \cX \to B$ is a good family of cubic
  surfaces parametrized by a $4$ dimensional proper variety $B$,
  then
  \[\int_{B} v_{1}^{2}v_{2} - v_{1}v_3 + 9v_{4} \geq 0\]
  with
  equality holding if and only if a general cubic surface does not
  arise as a fiber.
\end{corollary}

\begin{corollary}
  \label{cor:div} Let $\pi \colon \cX \to B$ is a good family of cubic
  surfaces parametrized by a $4$ dimensional proper variety $B$.
  Assume that a general fiber of $\pi$ is smooth and the induced map
  $\mu \colon B \dashrightarrow \mathcal M$ is dominant.
  Then the degree of $\mu$ is divisible by 1080.
\end{corollary}
In particular, a parameter space for cubic surfaces that carries a
good universal family must cover $\mathcal M$ by a map of degree
divisible by 1080.  Interestingly, the number 1080 is also the least
common multiple of the orders of the finite groups that arise as
automorphism groups of cubic surfaces.
We wonder if $1080$ is the minimal positive integer realizable as $\deg \mu$ as we vary
across all good families of cubic surfaces.

\subsection{Future directions}

\autoref{thm:eqvclass2} is only the beginning.  Each {\sl particular}
cubic surface $F=0$ registers its own equivariant orbit class
$[\Orb(F)]$.  A refined study of the methods developed in this paper
may afford us a complete understanding of how this class changes
according to special changes in the geometry of $F=0$, e.g. acquiring
singularities or Eckardt points.  This is something we hope to pursue
in the future.

In a different direction, there should be analogous formulas counting
Del Pezzo surfaces of degrees $1,2$ and $4$.  We fully expect our
methods to establish formulas in these settings, though we have not
even begun the analysis.

Finally, we repeat a comment of Aluffi and Faber \cite{alu.fab:93}: this paper computes the equivariant orbit class of a general vector in \emph{one} representation of \emph{one} group.  A whole lot of work needs to be done if we want to solve the problem in general.

\subsection{Organization}
In \autoref{sec:good}, we recall standard facts about cubic surfaces and their moduli.
We also review the basic notions of equivariant Chow theory.
The main new ingredient in this section is the definition of a good family (\autoref{def:goodfamily}).
In \autoref{sec:testfamilies}, we construct a number of test families of cubic surfaces and evaluate the linear relation on the coefficients of the class of the orbit closure given by each family.
For each family, we first give an informal description.
The first four families, taken up in the first four subsections, use variations on the theme of blowing up 6 points in $\P^2$.
The next four families, taken up in the fifth subsection, are isotrivial; their base space is the stack $B\G_m$.
Altogether, we produce 8 families for 5 undetermined coefficients.
In \autoref{sec:tietogether}, we tie together the enumerative results and obtain the main theorems stated in the introduction.
The last section, \autoref{sec:orbclosure} is devoted to developing methods to solve cases of the orbit closure problem.

\subsection{Acknowledgements}
We thank Radu Laza, who asked Question~\ref{slice} to the first two
authors during a meeting at Oberwolfach in the summer of 2016.  We
thank Hunter Spink for important discussions during the infancy
of the project.  We thank Paolo Aluffi for quickly providing an
enumerative reference on cuspidal cubic curves and Jarod Alper for discussions on equivariant geometry.


\section{Preliminaries}
\label{sec:good}

In this chapter, we introduce the notion of a good family of cubic
surfaces (\autoref{def:goodfamily}) and provide a tool
(\autoref{prop:good}) to verify that a family is good.  Before we do
so, we first establish notation and recall some facts about cubic surfaces and their moduli.
\subsection{Notation and conventions}
\label{sec:notation-conventions}
Fix an algebraically closed field $\k$ of characteristic zero.  In
this paper, a `scheme' means `a scheme of finite type over $\k$' and a
`point' means a `$\k$-point', unless explicitly stated otherwise.  We
do not distinguish between a vector bundle and the associated locally
free sheaf of its sections.

Given a map of schemes $\pi \colon X \to Y$ and a point $y \in Y$, we
let $X_y$ denote the fiber scheme $\pi^{-1}(y)$.  If $\pi$ is proper
and Gorenstein, we write $\omega_{X/Y}$ or $\omega_\pi$ for the
dualizing line bundle.  If $X$ and $Y$ are themselves Gorenstein and
$\pi$ is flat, then we have
$\omega_{X/Y} = \omega_X \otimes \pi^*\omega_Y^{-1}$.

For a closed subscheme $Y \subset X$, the notation $\mathcal I_Y \subset \O_X$ denotes the ideal sheaf of $Y \subset X$ and if $Y$ and $X$ are smooth, then $N_{Y/X}$ denotes the normal bundle of $Y \subset X$.

Given a vector bundle $\mathcal W$, we let $\P \mathcal W$ be its
projectivization.  Contrary to Grothendieck's convention,
$\P \mathcal W$ parametrizes $1$-dimensional subspaces in the fibers
of $\mathcal W$.  In other words, we set
\[ \P \mathcal W = \Proj \Sym(\mathcal W^\vee).\] As a result,
$\P \mathcal W$ comes equipped with a line bundle $\O(1)$ whose
push-forward to the base is $\mathcal W^\vee$.

Let $Q \subset \P V$ be a non-degenerate quadric.  The bilinear form
associated to $Q$ gives an isomorphism of $\P V$ with the dual
projective space $\P V^\vee$.  Given a point $p \in \P V$, we denote
by $\Polar_Q(p) \subset \P V$ the hyperplane dual to $p$ with respect to
$Q$. Similarly, given a hyperplane $H \subset \P V$, we denote by
$\Pole_Q(H) \in \P V$ the point dual to $H$ with respect to $Q$.  In
almost all cases, we use these constructions when $\P V$ is $\P^2$.

\subsection{Moduli of cubic surfaces}
\label{sec:classical-facts}
Let $\pi \colon \cX \to B$ be a family of cubic surfaces.  That is,
let $\pi$ be a flat, proper morphism whose fibers are isomorphic to
cubic hypersurfaces in $\P^3$.  Since $\pi$ is flat and its fibers are
Gorenstein, it is a Gorenstein morphism, and hence admits a dualizing
line bundle $\omega_\pi$.  By the adjunction formula, it follows that
the restriction of $\omega_{\pi}^{-1}$ to every fiber is isomorphic to
the restriction of $\O(1)$ with respect to an embedding of the fiber
as a cubic hypersurface in $\P^3$.  By standard theorems about
cohomology and base-change, it follows that
$R^i\pi_* \left(\omega_{\pi}^{-1}\right)$ vanishes for $i > 0$, and is
a vector bundle for $i = 0$.  Set
\[ \mathcal V = \pi_* \left( \omega_\pi^{-1} \right).\]
The natural map 
\[ \pi^*\mathcal V \to \omega_{\pi}^{-1}\]
yields a closed embedding
\begin{equation}\label{eqn:relcanx}
  \cX \subset \P \mathcal V^\vee.
\end{equation}
We know that
\[ \omega_{\P \mathcal V^\vee/B} = \O(-4) \otimes \det \mathcal V^\vee,\]
and hence, by adjunction,
\[ \omega_{\cX/B}= \O(-4) \otimes \O(\cX) \otimes \det \mathcal V^\vee |_\cX.\]
By construction, we also have
\[ \omega_{\cX/B} = \O(-1)|_\cX,\]
and hence we obtain that
\begin{equation}\label{eqn:ox}
  \O(\cX) = \O(3) \otimes \det \mathcal V.
\end{equation}
In other words, $\cX \subset \P \mathcal V^\vee$ is the zero locus of
a section of the line bundle $\O(3) \otimes \det \mathcal V^\vee$.

Consider the category $\mathscr {C}$ fibered in groupoids over the
category of schemes whose objects over a scheme $B$ are families of
cubic surfaces $\pi \colon \cX \to B$, and whose morphisms are pull-back
diagrams.  Consider the category $\mathscr {M}$ fibered in groupoids
over the category of schemes whose objects over a scheme $B$ are
vector bundles $\mathcal V$ over $B$ of rank 4 along with a nowhere
zero section $\xi$ of
$\Sym^3 \mathcal V^\vee \otimes \det \mathcal V$, and whose morphisms
are pull-back diagrams.  Let $V$ be a $\k$-vector space of dimension
4.  Note that $\mathscr{M}$ is simply the quotient stack
\[\mathscr{M} = \left[\Sym^3 V^\vee \otimes \det V \setminus \{0\} / \GL V \right]. \]
\begin{proposition}\label{prop:cubicstack}
  The categories $\mathscr{C}$ and $\mathscr{M}$ are equivalent as categories fibered over ${\rm Schemes}_k$.
\end{proposition}
\begin{proof}
  Let $\pi \colon \cX \to B$ be a family of cubic surfaces, and set
  $\mathcal V = \pi_* \left( \omega_{\pi}^{-1} \right)$.  Consider the
  relative canonical embedding $\cX \subset \P \mathcal V^\vee$.  In
  this embedding, $\cX$ is cut out by a section of
  $\O(3) \otimes \det \mathcal V$ on $\P \mathcal V^\vee$, or
  equivalently, by a section of $\xi$ of
  $\Sym^3V^\vee \otimes \det \mathcal V$ on $B$.  Since $\pi$ is flat,
  $\xi$ is nowhere zero on $B$.  We thus have a natural transformation
  from $\mathscr{C}$ to $\mathscr{M}$.

  To go from $\mathscr{M}$ to $\mathscr{C}$, given a rank 4 bundle $\mathcal V$ and a nowhere section $\xi$ of $\Sym^3 \mathcal V^\vee \otimes \det \mathcal V$, define $\cX \subset \P \mathcal V^\vee$ as the zero-locus of the corresponding section of $\O(3) \otimes \det \mathcal V$.
\end{proof}

We call $\mathscr{M}$ the \emph{moduli stack of cubic surfaces}.
Sometimes, it will be convenient to add in the zero section and consider the bigger stack
\[ \mathscr{M}^* = \left[\Sym^3 V^\vee \otimes \det V / \GL V \right].\]
Note that the diagonal $\G_m \subset \GL V$ acts by weight $-1$ on $\Sym^3 V^\vee \otimes \det V$.
In particular, its action is free.
Therefore, the stack $\mathscr{M}$ is isomorphic to the quotient stack
\[ \mathscr{M} = \left[ \P \Sym^3V^\vee / \PGL V \right].\]

We let $\mathcal M$ denote the GIT quotient of $\P \Sym^3V^\vee$ by
the action of $\PGL V$.  The semistable and stable locus of the
action were computed by Mumford; see \cite[Chapter 4, \S
2]{mum.fog.kir:94} or \cite[1.14]{mum:77}.  The stable locus consists
of surfaces with at worst $A_1$ singularities, and the semistable
locus includes surfaces with at worst $A_2$ singularities.  The ring
of invariants was computed much earlier, in the 1860s, by Salmon
\cite{sal:60} and Clebsch \cite{cle:61,cle:61*1}, and later by
Beklemishev \cite{bek:82} using more modern methods.  The invariant
ring is freely generated by polynomials of degree 8, 16, 24, 32, and
40, and as a result, the coarse moduli space $\mathcal M$ is
isomorphic to the weighted projective space $\P(1,2,3,4,5)$.

We recall perhaps the most famous fact about cubic surfaces, discovered in 1849. 
\begin{theorem}[Cayley \cite{cay:49}, Salmon \cite{sal:49}]
  \label{theorem:cayleysalmon}
  Every smooth cubic surface contains exactly $27$ lines.
\end{theorem}
As we know, the 27 lines are key to understanding geometry of the
cubic surface.  Each line is a $(-1)$-curve, and hence contractible by
Castelnuovo's theorem.  A set of lines that are pairwise skew is
particularly important, because they can be blown down simultaneously.
The maximum number of such lines is 6, and they allow us to realize
the cubic surface as a blow up of $\P^2$ at 6 points.

Following Schl\"afli \cite{sch:58}, we call an ordered tuple of pairwise disjoint six lines $(\ell_1, \dots, \ell_6)$ on a cubic surface a \emph{six}.
A \emph{double six} is an unordered pair of sixes $\{(\ell_1, \dots, \ell_6), (\ell_1', \dots, \ell_6')\}$ such that
\begin{itemize}
\item $\ell_{i}$ and $\ell'_{i}$ are disjoint for all $i$, and
\item $\ell_{i}$ and $\ell'_{j}$ meet at a point whenever $i \neq j$.
\end{itemize}

\begin{theorem}[Schl\"{a}fli \cite{sch:58}]
  \label{theorem:schlafli} Each six on a smooth cubic surface has a unique extension to a double-six, and every smooth cubic surface contains $36$ double-sixes.
\end{theorem}
Thus, each smooth cubic surface admits 72 distinct maps to
$\P^2$, up to automorphisms of $\P^2$, each of which
blows down 6 lines.

The sixes on a cubic surface allow us to relate $\mathcal M$ to a more
familiar moduli space.  Let $\mathcal M^\dagger$ denote the moduli
space of pairs $(S, \ell)$, where $S$ is a smooth cubic surface and
$\ell$ is a six on $S$.  By \autoref{theorem:schlafli}, the natural
forgetful map
\begin{align*}
  \sigma \colon \mathcal M^\dagger &\to \mathcal M \\
  (S, \ell) &\mapsto S
\end{align*}
is dominant and quasi-finite of degree
\[ \deg \sigma = 72 \times 6! = 51840.\] On the other hand, the space
$\mathcal M^\dagger$ is isomorphic to the configuration space of 6
ordered points on $\P^2$
\[ \mathcal M^\dagger \xrightarrow{\sim} \left( (\P^2)^6 \setminus
    \text{Diagonals} \right) / \Aut(\P^2).\] The isomorphism sends
$(S, (\ell_i))$ to $(\P^2, (p_i))$, where $\beta \colon S \to \P^2$
blows down the lines $\ell_i$ and $p_i = \beta(\ell_i)$.

\subsection{Equivariant Chow groups and equivariant orbits}
The standard reference for equivariant Chow theory is \cite{edi.gra:98}.
We recall a few relevant notions, mostly to set notation.

Given an algebraic space $X$ with an action of an algebraic group $G$,
denote by $A^i_G(X)$ the $G$-equivariant Chow group of codimension $i$
cycles.  Given a $G$-invariant subspace $Y \subset X$ of pure
codimension $i$, denote by $[Y] \in A^i_G(X)$ the fundamental class of
$Y$.  If $Y \subset X$ has codimension $c$ and $c > i$, then the restriction map
\[ A^i_G(X) \to A^i_G(X \setminus Y)\]
is an isomorphism; in general, it is surjective with kernel is $A^{i-c}_G(Y)$.
A $G$-linearized vector bundle
$V$ on $X$ admits equivariant Chern classes $c_i(V) \in A^i_G(X)$.
The equivariant Chow groups are homotopy invariant: if $Y \to X$ is a
$G$-equivariant vector bundle, then the pull-back map
$A^i_G(X) \to A^i_G(Y)$ is an isomorphism.  If $X$ is smooth, then
$\bigoplus_i A^i_G(X) = A^*_G(X)$ is a graded ring, where the multiplication is given by the
intersection product.

The $G$-equivariant Chow groups of $X$ are the same as the Chow groups
of the quotient stack $[X/G]$:
\[ A^i\left( [X/G] \right) = A^i_G(X).\]

The equivariant Chow ring of a point with the action of $G = \GL_n$ is
the polynomial ring
\[ A_{\GL_n}(\bullet) = \Z[c_1,\dots, c_n],\] where $c_i$ has
degree $i$ and can be identified with the $i$-th Chern class of the
standard $n$-dimensional representation of $\GL_n$.  By homotopy
invariance, for any $\GL_n$ representation $W$, we have
\[ A_{\GL_n}(W) = \Z[c_1,\dots,c_n].\]
If the group under consideration is $\GL V$, where $V$ is an $n$-dimensional vector space,
then we write $v_i$ for the $i$-th Chern class of the tautological
representation $V$.  Then we have
\[ A_{\GL V}(W) = \Z[v_1,\dots,v_n].\]

Let $V$ be a $4$-dimensional vector space.
Consider the $\GL V$ representation $W = \Sym^3 V^\vee \otimes \det V$.
By \autoref{prop:cubicstack}, (non-zero) $\GL V$ orbits in $W$ correspond canonically to cubic surfaces.
Note that the action of $\GL V$ on $W$ is generically free, and hence the codimension of a generic orbit is $20-16 = 4$.
Given a point $w \in W$, let $\Orb(w)$ denote the closure of its $\GL V$ orbit, and let
\[ [\Orb(w)] \in A^4_{\GL V}(W) = A^4(\mathscr M)\]
denote the fundamental class.
By a slight abuse of notation, if $X \subset \P^3$ is a cubic surface or $F \in k[X_0, \dots, X_3]$ is a non-zero cubic form, we let $\Orb(X)$ or $\Orb(F)$ denote the closure of the orbit in $W$ of the corresponding point $w$, and $[\Orb(X)]$ or $[\Orb(F)]$ the fundamental class.
Since $\A^4(\mathscr M)$ is generated by monomials in $v_i$ of total degree 4, we have an expression
\begin{equation}\label{eqn:chernexpansion}
  [\Orb(X)] = a_{1^4} \cdot v_1^4 + a_{1^2\cdot 2} \cdot v_1^2v_2 + a_{1\cdot 3} \cdot v_1v_3 + a_{2^2} \cdot v_2^2 + a_4 \cdot v_4,
\end{equation}
where the coefficients $a_i$ are integers.
The coefficients \eqref{eqn:chernexpansion} depend on $X$.
We let
\[ \partial \Orb(X) = \Orb(X) \setminus \GL V \cdot X,\]
denote the boundary of the orbit closure.
Note that $\partial \Orb(X) \subset W$ has codimension at least $5$.

\subsection{Good families}
\label{sec:good-families}
Our strategy is to compute $[\Orb(X)]$ for a generic $X$ by pulling it back to various families $B \to \mathscr M$.
If there is a point of $B$ that lies in every $\Orb(X)$, then we have the issue of excess intersection. 
A good family is one that does not suffer from this defect.
Recall that $V$ is a $4$-dimensional $\k$-vector space.
\begin{definition}[Good family]
\label{def:goodfamily}
A family of cubic surfaces $\pi \colon \cX \to B$ is \emph{good} if
the following holds: there exists a non-empty Zariski open subset
$U \subset \Sym^3V^\vee$ such that for every $F \in U$ and $b \in B$,
the fiber $\cX_b$ does not lie in $\partial \Orb(F)$.
\end{definition}
\begin{example}[Fibers with finite automorphism groups]
  \label{ex:finaut}
  Suppose every $\cX_b$ has a finite automorphism group.  Then
  $\pi \colon \cX \to B$ is good.  Indeed, let $U$ be the subset
  consisting of elements with finite automorphism group.  For every
  $F \in U$, the boundary $\partial \Orb(F)$ consists of elements with
  a positive dimensional automorphism group, and hence it cannot
  contain $\cX_b$.
\end{example}
\begin{example}[GIT semistable fibers]
  \label{ex:semistable}
  Suppose every fiber $\cX_b$ is GIT-semistable.
  Then the family $\pi \colon \cX \to B$ is good.
  Indeed, let $U$ be the subset consisting of GIT-stable elements.
  For every $F \in U$, the orbit of $F$ is closed in the semistable locus, and hence $\partial \Orb(F)$ cannot contain $\cX_b$.
\end{example}
\begin{example}[A combination]
  Suppose every fiber $\cX_b$ has a finite automorphism group or is GIT semistable.
  Let $U$ be the subset corresponding to smooth cubic surfaces.
  Since smooth cubic surfaces have a finite automorphism group and are GIT stable, we see that $\pi \colon \mathcal X \to B$ is good.
\end{example}

\begin{proposition}\label{prop:goodisgood}
  Let $\pi \colon \cX \to B$ be a good family of cubic surfaces.
  Assume that $B$ is proper of dimension 4 and a general fiber $\cX_b$ is semistable.
  Let $\mu \colon B \to \mathscr M$ be the map induced by $\pi$ and $m \colon B \dashrightarrow \mathcal M$ the rational map to the coarse moduli space.
  Then
  \[ \deg m = \deg \mu^*\left( \Orb(F) \right)\]
  for a general cubic form $F$.
\end{proposition}
\begin{proof}
  There exists an open $W \subset \mathcal M$ such that $m^{-1} (W) \to W$ is a finite \'etale morphism.
  Then the degree of $m$ is simply the number of points in $m^{-1}(u)$ for $u \in W$.
  By shrinking $W$ if necessary, we may assume that the map
  \[ \mathscr M^{ss} \to \mathcal M\]
  is an isomorphism over $W$.

  Let $U$ be as in the definition a good family.
  Then, for every $u \in U$, we have the equality
  \[ \cX \times_{\mathscr M} u = \cX \times_{\mathscr M} \overline u.\]
  By shrinking $U$ if necessary, we may assume that $U$ consists of semistable points and maps to $W \subset \mathcal M$ in the coarse space.
  Then, for every $u \in U$, we have
  \begin{align*}
    \deg \mu^* \Orb(u) &= \# \left(\cX \times_{\mathscr M} u \right)\\
                       &= \# \left(\cX \times_{\mathcal M} u\right) \\
                       &= \deg m.
  \end{align*}
\end{proof}

\subsection{Cubic surfaces as blow-ups}
Cubic surfaces arise as blow ups of $\P^2$ at six points, or more generally, as blow ups of del Pezzo surfaces of degree $d$ in $(d-3)$ points.
Varying the center of blow up is a convenient source of families of cubic surfaces.
The following proposition identifies certain nice sub-schemes to blow up.

\begin{definition}[Admissible subscheme]
  \label{def:admissible}
  Let $S$ be a smooth del Pezzo surface of degree $d \geq 3$ and let $Z \subset S$ be a subscheme of length $(d-3)$.
  We say that $Z$ is \emph{admissible} if 
  \begin{enumerate}
  \item $Z$ is curvilinear (contained in a smooth curve),
  \item $h^{0}\left(\mathcal{I}_Z \otimes \omega_S^{-1}\right) = 4$, and
  \item $\mathcal I_Z \otimes \omega_S^{-1}$ is generated by its global sections.
  \end{enumerate}
\end{definition}
The last condition means that $Z$ is scheme-theoretically cut out by the sections of $\omega_S^{-1}$ that vanish on it.

\begin{proposition}
  \label{prop:admissibleblowup}
Let $Z \subset S$ be an admissible length $(d-3)$ subscheme.
Set $X = \Bl_ZS$.
Then the following hold.
\begin{enumerate}
\item $X$ has at worst $A_n$ singularities.
\item $H^0\left(X, \omega^{-1}_X\right)$ is $4$-dimensional and base-point free.
\item $H^i(X,\omega_X^{-n}) = 0$ for all $n \geq 0$ and $i > 0$.
\item The image of the map $X \to \P^3$ induced by $\omega_X^{-1}$ is a cubic surface $Y$; the map $X \to Y$ is birational.
\item Let $S = \bigoplus_{n \geq 0} S_n$ be the homogeneous coordinate ring of $Y$.
  The pull-back map $S_n \to H^0\left(X, \omega_X^{-n}\right)$ is an isomorphism.
\item $Y$ is normal and has rational double point (ADE) singularities.
\end{enumerate}
\end{proposition}
\begin{proof}
  Let $\pi \colon X \to S$ be the blow-up map.
  Since $Z$ is curvilinear, $X$ has at worst $A_n$ singularities.
  Indeed, if $Z$ is isomorphic to $\spec k[t]/t^{n+1}$ at a point $z$ in its support, then $X$ is singular at a unique point of $\pi^{-1}(z)$; this singularity is of type $A_n$.
  Note that $E_z := \pi^{-1}(z)$ (taken with the reduced structure) is isomorphic to $\P^1$.
  
  Let $E \subset X$ be the exceptional divisor, namely, the Cartier divisor defined by the ideal $\pi^{-1}\mathcal I_Z$.
  Over the point $z$ as above, $E$ has multiplicity $n$ along the component $E_z$.
  Furthermore, we have $(nE_z)^2 = -n$, and hence $E^2 = -(d-3)$.
  Also, see that we have
  \begin{equation}\label{eqn:wx}
    \omega_X = \pi^* \omega_{S} \otimes \O_X(E),
  \end{equation}
  and therefore,
  \[ \omega_X^{-1} = \pi^* \omega^{-1}_{S} \otimes \O_X(-E).\]
  By applying $\pi_*$, we get
  \[ \pi_* \left( \omega_X^{-1} \right) = \mathcal I_Z \otimes \omega_S^{-1}.\]
  By our assumption, $H^0\left(\mathcal I_Z \otimes \omega_S^{-1}\right)$ is $4$-dimensional and its base-locus is precisely $Z$.
  Hence $H^0\left( X, \omega_X^{-1} \right)$ is $4$-dimensional and base-point free.

  For the vanishing of higher cohomology, we induct on $n$.
  Since $X$ is a rational surface with rational singularities, we have $h^1(\O_X) = h^1(\omega_X) = 0$.
  Let $C \subset X$ be the zero locus of a general section of $\omega_X^{-1}$.
  By Bertini's theorem, $C$ is smooth and by the adjunction formula, it has arithmetic genus $1$.
  Taking global sections in the sequence
  \[ 0 \to \omega_X \to \O_X \to \O_C \to 0\]
  shows that $h^0(\O_C) = 1$, and hence $C$ is connected.
  Twisting the above sequence by $\omega_X^{-n}$ gives
  \[ 0 \to \omega_X^{-n+1} \to \omega_X^{-n} \to \omega_X^{-n}|_C \to 0.\]
  The restriction $\omega_X^{-n}|_C$ has degree $3n$, which is positive if $n$ is positive, and hence $h^i\left(\omega_X^{-n}|_C\right) = 0$ for $i > 0$.
  By applying the long exact sequence in cohomology, we see that if $i > 0$, then $h^i(\omega_X^{-n+1}) = 0$ implies $h^i(\omega_X^{-n}) = 0$.
  By induction on $n$, we conclude $h^i(\omega_X^{-n}) = 0$ for all $n \geq 0$ and $i > 0$.
    
  To get the degree of the image of $X \to \P^3$, we compute the self-intersection $c_1(\omega_X)^2$.
  Using \eqref{eqn:wx} and $E^2 = -(d-3)$, we compute
  \[ c_1(\omega_X)^2 = c_1(\omega_{S})^2 - d + 3 = 3.\]
  Thus, either the image is a hyperplane $\Pi$ and the map $X \to \Pi$ has degree 3, or the image is a cubic surface $Y$ and the map $X \to Y$ is birational.
  We rule out the first possibility since the map $X \to \P^3$ is defined by linearly independent sections, and hence does not factor through a hyperplane.

  By Riemann--Roch, we get that $\chi(X,\omega_X^{-n}) = \chi(Y, \O(n))$.
  We also know that for $i > 0$ and $n \geq 0$, both $h^i(X, \omega_X^{-n})$ and $h^i(Y, \O(n))$ vanish.
  Hence, we get $h^0(X, \omega_X^{-n}) = h^0(Y, \O(n))$.
  Recall that, the $n$-th graded component $S_n$ of the homogeneous coordinate ring of $Y$ is precisely $H^0(Y, \O(n))$.
  Since $X \to Y$ is surjective, the natural map $S_n \to H^0(X, \omega_X^{-n})$ is injective.
  Since the dimensions of the source and the target are equal, it must be an isomorphism.

  From the isomorphism $S \cong \bigoplus_{n \geq 0} H^0(X, \omega_X^{-n})$, we get
  \[ Y = \Proj\left( \bigoplus_{n \geq 0}H^0\left(X, \omega_X^{-1} \right) \right).\]
  Since $X$ is normal, so is $Y$.
  It is known that an irreducible, normal, rational cubic surface, such as $Y$, only has rational double point (ADE) singularities \cite{bru.wal:79}.
\end{proof}

\begin{definition}[Associated cubic surface]
  \label{def:cubicZ}
  Let $S$ be a smooth del Pezzo surface of degree $d$ and let $Z \subset S$ be an admissible subscheme of length $d-3$.
  The \emph{cubic surface associated to $Z \subset S$}, denoted by
  $X_Z \subset \P^3$, is the image of the map
  $\Bl_{Z}S \to \P^3$ induced by the anti-canonical linear system $H^0\left(\Bl_ZS, \omega_{\Bl_Z S}^{-1}\right)$.
\end{definition}
By \autoref{prop:admissibleblowup}, we may equivalently define $X_Z$ as
\[ X_Z = \Proj \left( \bigoplus_{n \geq 0} H^0\left( \Bl_ZS, \omega_{\Bl_Z S}^{-n} \right) \right).\]
Let us make the construction $Z \leadsto X_Z$ in families.
Let $B$ be a scheme, $P \to B$ a smooth proper family of del-Pezzo surfaces and $Z \subset  P$ a closed subscheme, flat over $B$, whose fibers are admissible subschemes.
Consider the blowup $\Bl_{ Z}  P$.
Note that $Z \subset  P$ is a local complete intersection, and hence the blow-up is $B$-flat and commutes with arbitrary base change.
In particular, the fiber $(\Bl_ZP)_b$ is the blow up of $P_b$ along $Z_b$.
(The statement about blow-ups and base-change must be well-known, but we could not find a reference so we include a proof; see \autoref{prop:blow-ups-and-base-change}.)
The family of cubic surfaces $X_Z \to B$ associated to $Z \subset P \to B$ is defined by
\[
  X_Z = \Proj_B \left(\bigoplus_{n \geq 0}\pi_*\left(\Bl_ZP, \omega_{\Bl_ZP/B}^{-n}\right)\right).
\]
By \autoref{prop:admissibleblowup}, we see that $X_Z$ is flat over $B$
and its fiber over $b \in B$ is the cubic surface associated to
$Z_b \subset P_b$, as in \autoref{def:cubicZ}.
\begin{proposition}\label{prop:blow-ups-and-base-change}
  Let $\pi \colon P \to B$ be a smooth morphism and $Z \subset P$ a closed subscheme flat and lci over $B$.
  Then $\Bl_ZP \to B$ is flat and for any $C \to B$, the canonical map
  \[  \Bl_{Z_C} P_C \to \left(\Bl_ZP\right)_C.\]
  is an isomorphism.
\end{proposition}
\begin{proof}
  The statement is local on $P$.
  After passing to a suitable affine open, we are reduced to the following situation in algebra.
  We have a Noetherian ring $B$ and a finitely generated smooth $B$-algebra $A$.
  We have an ideal $I \subset A$ such that $A/I$ is $B$-flat and $I$ is generated by a regular sequence.
  The blow-up is the $\Proj$ of the Rees algebra
  \[R_I(A) = \bigoplus_{n \geq 0}I^n.\]
  Since $I \subset A$ is generated by a regular sequence, the associated graded ring $\operatorname{gr} R_I(A)$ is a polynomial algebra over $A/I$, and hence flat over $B$.
  It follows that $R_I(A)$ is flat over $B$, and hence its $\Proj$ is flat over $\spec B$.
  
  For the statement about base-change, we show that the formation of the Rees algebra commutes with base-change.
  For a $B$-algebra $C$, let $J \subset A \otimes_B C$ be the image of $I \otimes_B C$.
  Consider the map
  \[ i \colon R_I(A) \otimes_B C \to R_J(C).\]
  We use again that the associated graded rings are polynomial algebras, on which the map induced by $i$ is clearly an isomorphism.
  Hence $i$ is an isomorphism.
\end{proof}

Let $Z \subset \P^2$ be an admissible subscheme of length 6.
An automorphism of the pair $(\P^2, Z)$ induces an automorphism of $\Bl_Z\P^2$ and hence an automorphism of the associated cubic surface $X_Z$.

\begin{proposition}
  \label{prop:good} Let $Z \subset \P^2$ be an admissible, length $6$
  subscheme and $X_Z \subset \P^3$ its associated cubic surface.
  The map
  \[ \Aut(\P^2, Z) \to \Aut(X_Z)\]
  is injective and its image has finite index.
  In particular, $\Aut(\P^2, Z)$ is finite if and only if $\Aut(X_Z)$ is.
\end{proposition}

\begin{proof}
  Since we have a birational isomorphism $\P^2 \dashrightarrow X_Z$, a non-trivial automorphism of $\P^2$ cannot induce a trivial automorphism of $X_Z$.
  Hence the map of automorphism groups is injective.

  To show that the image has finite index, let $\tau \colon \widetilde Y \to \Bl_Z\P^2$ be the minimal desingularization and set $Y = X_Z$.
  Since $\Bl_Z\P^2$ has only $A_n$ singularities and such singularities are canonical, we conclude that
  \[\tau^* \omega_{\Bl_Z\P^2} = \omega_{\widetilde Y}.\]
  Consider the composite $\mu$ of 
  \[ \widetilde Y \xrightarrow{\tau} \Bl_Z \P^2 \xrightarrow{\kappa} Y.\]
  Note that
  \[ \mu^* \omega_Y = \omega_{\widetilde Y}.\]
  We claim that $\mu \colon \widetilde Y \to Y$ is the minimal desingularization of $Y$.
  To see this, let $Y' \to Y$ be the minimal desingularization.
  Since $\widetilde Y$ is non-singular, the birational morphism $\mu \colon \widetilde Y \to Y$ factors as a composite
  \[ \widetilde Y \xrightarrow{\alpha} Y' \to Y.\]
  Since $\alpha$ is a birational morphism between non-singular surfaces, it is a composite of blow-ups at points.
  In particular, if it is not an isomorphism, then there exists a $(-1)$-curve $E \subset \widetilde Y$ contracted by $\alpha$.
  But then we have 
  \[ E \cdot \omega_{\widetilde Y} = E \cdot \mu^* \omega_Y = 0,\]
  which contradicts the adjunction formula
  \[ \left(\omega_{\widetilde Y}+ E\right) \cdot E = -2.\]
  Thus, we deduce that $\alpha$ is an isomorphism, and hence $\widetilde Y \to Y$ is the minimal desingularization of $Y$.
  
  Since every automorphism of a surface lifts uniquely to its minimal
  desingularization, we have an injective map
  \[ \Aut(\P^3, X_Z) \to \Aut(\widetilde Y).\]
  It suffices to show that the image of $\Aut(\P^2, Z)$ has finite index in $\Aut(\widetilde Y)$.

  Let $H \in \Pic(\widetilde Y)$ be the pullback of $\O(1)$ under the map $\widetilde Y \to \P^2$.
  We claim that the orbit of $H$ under $\Aut(\widetilde Y)$ is finite.
  To see this, note that we have
  \[ \omega_{\widetilde Y}^2 = 3 \text{ and } \omega_{\widetilde Y} \cdot H = -3, \]
  and hence $(\omega_{\widetilde Y} + H)$ is orthogonal to $\omega_{\widetilde Y}$ under the intersection pairing.
  Since $\omega_{\widetilde Y}$ has positive self-intersection, its orthogonal complement $\omega_{\widetilde Y}^\perp \subset \Pic(\widetilde Y) = {\rm NS}(\widetilde Y)$ is a negative definite lattice by the Hodge index theorem.
  Thus, it has only finitely many elements of a given self-intersection.
  Any automorphism of $\widetilde Y$ preserves $\omega_{\widetilde Y}$ and the self-intersection number.
  Therefore, under $\Aut(\widetilde Y)$, the orbit of $\omega_{\widetilde Y} + H$ and hence the orbit of $H$, must be finite.

  Since the orbit of $H$ is finite, its stabilizer is a finite index subgroup.
  We claim that this subgroup is precisely the image of $\Aut(\P^2, Z)$.
  Let $\tau \colon \widetilde Y \to \widetilde Y$ be an automorphism fixing $H$.
  The map $\beta \colon \widetilde Y \to \P^2$ is given by the complete linear system $|H|$.
  Hence $\tau$ induces an automorphism $\tau' \colon \P^2 \to \P^2$.
  It remains to show that $\tau'$ preserves $Z$.
  Note that we have
  \[ \beta_* \left(\omega_{\widetilde Y}^{-1}\right) = \O(3) \otimes \mathcal I_Z,\]
  and hence
  \[ \beta_* \left( \omega_{\widetilde Y}^{-1} (-3H) \right) = \mathcal I_Z.\]
  Therefore, the induced map $\tau'$ preserves $\mathcal I_Z$, and hence $Z$.
\end{proof}


\section{Test families}\label{sec:testfamilies}
\subsection{The first test family}
\label{sec:family-b_1}
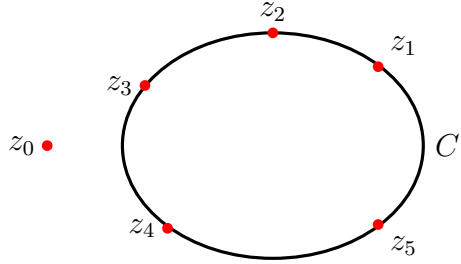
\begin{figure}
  \centering
  \begin{tikzpicture}[very thick]
  \draw (0,0) ellipse (2 and 1.5) (2,0) node [right] {$C$};
  \draw[red,fill] (-3, 0) circle (0.05) node [black, left] {$z_0$};
  \draw[red, fill] (1.4, 1.05) circle (0.05) node [black, above right] {$z_1$}
  (0,1.5) circle (0.05) node [black, above] {$z_2$}
  (-1.7,.8) circle (0.05) node [black, left] {$z_3$}
  (-1.4,-1.1) circle (0.05) node [black, left] {$z_4$}
  (1.4,-1.05) circle (0.05) node [black, below right] {$z_5$};
\end{tikzpicture}

  \caption{In the first test family, we blow up the plane along $\{z_0, \dots, z_5\}$ as $z_0$ stays fixed and $z_1,\dots,z_5$ vary in a general 4-dimensional series on a fixed conic $C$.}
  \label{fig:FamilyB1}
\end{figure}
Fix a point $z_0$ and a conic $C \subset \P^2$ not containing $z_0$.
We vary 5 points $z_1, \dots, z_5$ in a general 4-dimensional linear series on $C$.
Blowing up the six points $\{z_0, \dots, z_5\}$ gives a 4-parameter family of cubic surfaces; see \autoref{fig:FamilyB1}.
\begin{proposition}
  \label{prop:B1}
  Let $C$ be a plane conic, $Z \subset C$ a subscheme of length 5, and $z_0 \in \P^2$ a point not on $C$.
  Then the subscheme $Z \cup \{z_0\} \subset \P^2$ is admissible.
\end{proposition}
\begin{proof}
  The subscheme $Z' = Z \cup \{z_0\}$ is clearly curvilinear.
  Checking that $\mathcal I_{Z'} \otimes \O(3)$ is generated by its global sections is straightforward.
\end{proof}

We now formally define the family.
Let $H = C^{[5]}$ be the Hilbert scheme of 5 points in $C$, so that $H \cong \P^5$, and let $B \cong \P^4 \subset H \cong \P^5$ be a general hyperplane.
Let $Z \subset B \times C$ be the restriction of the universal closed subscheme.
Set 
\[Z' = Z \sqcup \left(B \times \{z_0\}\right) \subset B \times \P^2.\]
Our family
\[\pi \colon X \to B\]
is the family of cubic surfaces associated to $Z' \subset B \times \P^2$ (see \autoref{def:cubicZ}).
\begin{proposition}\label{prop:goodnessB1}
  The family $\pi \colon X \to B$ is good.
\end{proposition}
\begin{proof}
  We show that for every $b \in B$, the group $\Aut(X_b)$ is finite (see \autoref{ex:finaut}).
  By \autoref{prop:good}, it suffices to show that the group $\Aut(\P^2, Z_b')$ is finite.
  Since the support of $Z_b'$ is finite, it suffices to show that the subgroup of $\Aut(\P^2, Z'_b)$ that acts trivially on the support is finite.
  Let $\sigma \in \Aut(\P^2, Z'_b)$ act trivially on the support.
  Since $C$ is the unique conic containing $Z_b$, the element $\sigma$ must preserve $C$.
  Since $\sigma(z_0) = z_0$, and $\sigma$ preserves $C$, it must fix the pair of points $\Polar_C(z_0) \cap C$.
  Since $B \subset \Hilb^5C$ is general, we may assume that it does not contain the finitely many length 5 schemes supported on $\Polar_C(z_0) \cap C$.
  Then $\sigma$ fixes at least 3 points of $C$, namely the two points $\Polar_C(z_0) \cap C$ and the points of $Z \subset C$.
  It follows that $\sigma$ fixes all points of $C$ and hence all points of $\P^2$.
\end{proof}

We now compute with the Chern classes of the anti-canonical section bundle.
\begin{proposition}\label{lem:v1}
  Let $\mathcal V = \pi_* \left( \omega_{X/B}^{-1} \right)$.
  Then we have
  \[
    \int_{B}v_{1}^{4} = 16, \quad \int_{B}v_{1}^{2}v_{2} = 4,\quad \int_{B}v_{1}v_{3} =0,\quad \int_{B}v_2^{2} = 1,\quad \int_{B}v_{4} = 0.\]
\end{proposition}
\begin{proof}
By the construction of the cubic surface, we see that $\mathcal V$ is the rank 4 bundle
\[ \mathcal V = \pi_* (\mathcal I_{Z'} \otimes \omega_{\pi}^{-1}),\]
where $\pi \colon B \times \P^2 \to B$ is the first projection.
It contains the constant rank 2 bundle
\begin{equation}\label{eqn:v11}
  \pi_*\left( \mathcal I_{B \times C} \otimes \mathcal I_{B \times z_0} \otimes \omega_\pi^{-1} \right) = H^0\left(\mathcal I_{C} \otimes \mathcal I_{z_0} \otimes \omega_{\P^2}^{-1}\right) \otimes \O_{B},
\end{equation}
and the quotient is the rank 2 bundle $\pi_*\left( I_{Z/B \times C} \otimes \omega_\pi^{-1} \right)$.
To identify the quotient, note that $Z \subset B \times C$ is a divisor of class $(1,5)$ and $\omega^{-1}_{\pi}|_{B\times C}$ is isomorphic to $\O(0,6)$.
Hence, $I_{Z/B \times C} \otimes \omega_\pi^{-1}$ is isomorphic to $\O(-1,1)$, and therefore
\begin{equation}\label{eqn:v12}
  \begin{split}
    \pi_*\left( I_{Z/B \times C} \otimes \omega_\pi^{-1} \right) \cong \O_{B}(-1)^{2}.
  \end{split}
\end{equation}
By combining \eqref{eqn:v11} and \eqref{eqn:v12}, we see that $\mathcal V$ is an extension
\begin{equation}\label{eqn:V1}
  0 \to \O_{B}^2 \to \mathcal V \to \O_{B}(-1)^2 \to 0.
\end{equation}
The Chern class calculation is now straightforward using the Whitney sum formula.
\end{proof}

Let us find the degree of the map to moduli $B \dashrightarrow \M$.
\begin{proposition}\label{lem:deg1}
  The degree of $B \dashrightarrow \M$ is $4320$.
\end{proposition}
\begin{proof}
  Consider the space of marked cubic surfaces $\M^\dagger$, which is isomorphic to the configuration space of 6 ordered points on $\P^2$.
  We have the action of $S_5$ on $\M^\dagger$ via permutations of the last 5 points.
  The map $B \dashrightarrow \M$ lifts to
  \begin{equation}\label{eqn:b1lift}
    B \to \M^\dagger / S_5,
  \end{equation}
  where $b \in B$ is mapped to $(z_0, z_1 + \dots + z_5)|_b$.
  Let us compute the degree of this map.
  Given a general $(y_0, y_1 + \dots + y_5)$ in $\M^\dagger / S_5$, there is a unique conic $Q$ through $y_1 + \dots + y_5$.
  Note that there exists an automorphism of $\P^2$ that takes $(C, z_0)$ to $(Q, y_0)$, and such an automorphism is unique up to automorphisms preserving $(C, z_0)$.
  Suppose, under one such automorphism, $y_1 + \dots + y_5 \subset Q$ is taken to $z_1' + \dots + z_5' \subset C$.
  We must find out how many points of $B$ are equivalent to $z_1' + \dots + z_5'$ up to an automorphism of $\P^2$ preserving $(C, z_0)$.
  Let $w_0 + w_\infty \subset C$ be the scheme $\Polar_C(z_0) \cap C$.
  Then we have
   \[\Aut(\P^2, C, z_0) = \Aut(C, w_0 + w_{\infty}) \cong \G_m \rtimes \Z/2\Z. \]
   Choose an isomorphism $C \cong \P^1$ such that $w_0$ and $w_\infty$ are identified with $[0:1]$ and $[1:0]$.
   Then the $\G_m$ corresponds to the diagonal $\G_m \subset \PGL_2$ and the $\Z/2\Z$ corresponds to the swapping of the two homogeneous coordinates.
   Write $z_1' + \dots + z_5'$ as the vanishing locus of a binary quintic form $F$.
   It is easy to see that under the action of the $\G_m$, the closure of the orbit of $V(F)$ in $C^{[5]} \cong \P^5$ is a rational normal curve.
   Hence, a general hyperplane $B$ intersects the orbit in 5 points.
   Accounting for the additional $\Z/2\Z$, we conclude that there are precisely 10 points of $B$ equivalent to $z_1' + \dots + z_5'$ up to an automorphism of $\P^2$ preserving $(C, z_0)$.
   Thus, the degree of the map in \eqref{eqn:b1lift} is 10.

   Since the degree of $\M^\dagger \to \M$ is $72 \times 6!$, the degree of $B \to \M$ is
   \[72 \times 6! / 5! \times 10 = 4320.\]
\end{proof}

Using \autoref{prop:goodisgood} and the computation of the Chern classes and the degree, we obtain the following relation on the undetermined coefficients of $[\Orb(X)]$ in
\eqref{eqn:chernexpansion} for a general cubic surface $X$:
\begin{align}
  \label{eq:relation1}
  16 \cdot a_{1^4} + 4 \cdot a_{1^2 2} + a_{2^2} = 4320.
\end{align}


\subsection{The second test family}
\label{sec:family-b_2}
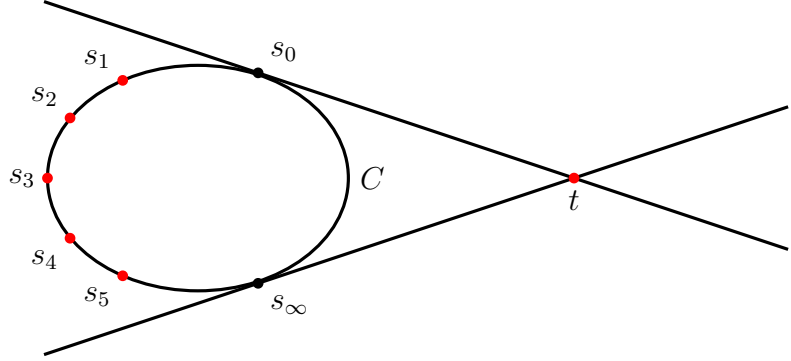
\begin{figure}
    \centering
    \begin{tikzpicture}[very thick]
  \draw (0,0) ellipse (2 and 1.5) (2,0) node [right] {$C$};
  \draw[shorten >=-3cm, shorten <=-3cm] (0.8, 1.4) -- (5,0);
  \draw[shorten >=-3cm, shorten <=-3cm] (0.8, -1.4) -- (5,0);

  \draw[red, fill] (5,0) circle (0.05) node [black, below] {$t$};
  \draw[fill] (0.8, 1.4) circle (0.05) node [above right] {$s_0$};
  \draw[fill] (0.8, -1.4) circle (0.05) node [below right] {$s_\infty$};

  \draw[red, fill]
  (-1, 1.3) circle (0.05) node [black, above left] {$s_1$}
  (-1.7, 0.8) circle (0.05) node [black, above left] {$s_2$}
  (-2, 0) circle (0.05) node [black, left] {$s_3$}
  (-1, -1.3) circle (0.05) node [black, below left] {$s_5$}
  (-1.7, -0.8) circle (0.05) node [black, below left] {$s_4$};
\end{tikzpicture}

    \caption{
      In the second test family, we blow up the plane at the points $\{s_{1}, \dots, s_{5}, t\}$ as the seven points $\{s_1, \dots, s_5, s_0, s_\infty\}$ vary freely on a conic $C$.}
    \label{fig:FamilyB2}
  \end{figure}
Forget the notation introduced in \autoref{sec:family-b_1}.
We take a smooth conic $C \subset \P^2$ and 7 distinct marked points $s_0, s_1, \dots, s_5, s_\infty$ on $C$.
Let $t \in \P^2$ be the pole, with respect to $C$, of the line joining $s_0$ and $s_\infty$.
We blow up the six points $s_1,\dots, s_5, t$ to get a cubic surface (see \autoref{fig:FamilyB2}).
As the 7 points vary on $C$, we get a 4-parameter family of cubic surfaces.

We now make the construction more precise, and show how to use it construct a family of cubic surfaces over $\overline {\M}_{0,7}$.
Let $\phi \colon \mathcal C \to \overline \M_{0,7}$ be the universal curve with the universal sections denoted by $\sigma_0, \sigma_1, \dots, \sigma_5, \sigma_\infty$.
Let $\mathcal L$ be the line bundle $\sigma_\infty^*\left(\omega_\phi\right)$ on $\overline \M_{0,7}$.
As is customary, we denote its Chern class by $\psi_\infty$.
Consider the rank 2 vector bundle
\[ \mathcal A = \phi_* \O_{\mathcal C}\left(\sigma_\infty\right).\]
It is easy to check that the evaluation maps yield an isomorphism
\begin{align*}
  \mathcal A &\xrightarrow{\sim} \sigma_\infty^*\O_{\mathcal C}\left( \sigma_\infty \right) \oplus \sigma_0^*\O_{\mathcal C}\left( \sigma_\infty \right)\\
  &= \mathcal L^{-1} \oplus \O_{\overline \M_{0,7}}.
\end{align*}
Set $\overline {\mathcal C} = \P \mathcal A^\vee$ with the structure map $\overline \phi \colon \overline {\mathcal C} \to \overline \M_{0,7}$.
The evaluation map $\phi^*\mathcal A \to \O_{\mathcal C}\left(\sigma_\infty\right)$ yields a map $\gamma \colon \mathcal C \to \overline {\mathcal C}$.
The effect of $\gamma$ on a stable curve $(C, s_0, \dots, s_\infty)$ is to contract all irreducible components of $C$ not containing the marked point $s_\infty$.
The unique remaining $\P^1$ is identified with the corresponding fiber of $\overline \phi$.
Set $\overline \sigma_i = \sigma_i \circ \gamma$; these $\overline \sigma_i$ are sections of $\overline \phi \colon \overline {\mathcal C} \to \overline \M_{0,7}$.
They satisfy the following properties
\begin{itemize}
\item $\overline{\sigma}_{\infty}$ is disjoint from all the other sections, and
\item in each fiber of $\overline \phi$, the union of all seven sections
  is supported on at least three points.
\end{itemize}
We now invoke the relative 2-Veronese embedding
\[
\iota \colon \overline {\mathcal C} \to \mathcal P := \P \Sym^2 \mathcal A^\vee.
\]
Let $\rho \colon \mathcal P \to \overline{\M}_{0,7}$ be the structure morphism.
Then $\overline {\mathcal C} \subset \mathcal P$ is a family of smooth conics with 7 marked points $\sigma_0, \dots, \sigma_\infty$ satisfying the two properties above.

Let $\tau \colon \overline \M_{0,7} \to \mathcal P$ be the section obtained by fiber-wise taking the pole with respect to $\overline {\mathcal C}$ of the line joining the points $\overline \sigma_0$ and $\overline \sigma_\infty$.
Here is an alternate description of $\tau$.
The direct sum decomposition $\mathcal A = \O \oplus \mathcal L^{-1}$ gives the decomposition
\[ \Sym^2 \mathcal A = \O \oplus \mathcal L^{-1} \oplus \mathcal L^{-2}.\]
In $\P \mathcal A^\vee$, the sections $\overline \sigma_0$ and $\overline \sigma_\infty$ correspond to the projections $\mathcal A \to \O$ and $\mathcal A \to \mathcal L^{-1}$, respectively.
Hence, in the Veronese embedding, they correspond to the projections $\Sym^2 A \to \O$ and $\Sym^2 A \to \mathcal L^{-2}$, respectively.
Using the quadratic form on $\Sym^2 \mathcal A$ given by the Veronese conic, it is easy to check that the section $\tau$ is given by the third projection $\Sym^2 A \to \mathcal L^{-1}$.

We let $\Sigma \subset \overline{\mathcal C}$ denote the closed subscheme associated to the divisor $\overline \sigma_1 + \dots + \overline \sigma_5$ and $\mathcal Z \subset \mathcal P$ the disjoint union $\Sigma \sqcup \tau$.
By \autoref{prop:B1}, $\mathcal Z \subset \mathcal P$ is a family of admissible subschemes.
We let
\[ \pi \colon \mathcal X \to \overline \M_{0,7}\]
be the family of cubic surfaces associated to $\mathcal Z \subset \mathcal P$ (see \autoref{def:cubicZ}).

\begin{proposition}  \label{prop:goodnessB2}
  The family $\pi \colon \mathcal X \to \overline \M_{0,7}$ is good.
\end{proposition}
\begin{proof}
  All fibers have finite automorphism groups by the same proof as \autoref{prop:goodnessB1}.
\end{proof}

We now compute the Chern classes of the anti-canonical section bundle.
\begin{proposition}\label{lem:v2}
  Let $\mathcal V = \pi_* \omega_{\pi}^{-1}$ and let $v_i = c_i(\mathcal V)$.
  Then we have
  \[
    \begin{split}
      \int_{\overline \M_{0,7}}v_{1}^{4} = 625&, \quad \int_{\overline \M_{0,7}}v_{1}^{2}v_{2} = 125,\quad \int_{\overline \M_{0,7}}v_{1}v_{3} = -25,\\
      &\int_{\overline \M_{0,7}}v_2^{2} = 25,\quad \int_{\overline \M_{0,7}}v_{4} = -6.
    \end{split}
  \]
\end{proposition}
For the proof, we use the following.
\begin{lemma}\label{lem:push}
  Let $\pi \colon P \to B$ be a $\P^1$-bundle with a section $\sigma$.
  Let $D \subset P$ be a divisor of relative degree $d$ with respect to $\pi$ with $d \geq -1$.
  Then, in the Grothendieck group of $B$, we have
  \[ [\pi_* \O(D)] = [\sigma^*\O(D)] + \dots + [\sigma^*\O(D-(d-1)\sigma)].\]
\end{lemma}
\begin{proof}
  We induct on $d$.
  Both sides are 0 for $d = -1$.
  The exact sequence
  \[ 0 \to \O(D-\sigma) \to \O(D) \to \O(D)|_\sigma \to 0\]
  pushed-forward to $B$ provides the induction step.
\end{proof}
\begin{proof}[Proof of \autoref{lem:v2}]
  Recall that $\mathcal X \to \overline \M_{0,7}$ is the cubic surface associated to $\mathcal Z \subset \mathcal P$ in the $\P^2$-bundle $\rho \colon \mathcal P \to \overline \M_{0,7}$.
  By construction, we have
  \[ \mathcal V = \rho_*(\mathcal I_{\mathcal Z} \otimes \omega^{-1}_\rho).\]
  Hence, in the Grothendieck group of $\overline \M_{0,7}$, we have
  \begin{equation}\label{eqn:diff}
    [\mathcal V] = \rho_*(\omega_\rho^{-1}) - \rho_*(\omega_\rho^{-1}|_{\mathcal Z}).
  \end{equation}
  By the relative Euler sequence, we see that $\omega^{-1}_\rho$ is isomorphic to $\O_{\mathcal P}(3) \otimes \rho^* \det (\Sym^2 \mathcal A)^\vee$, and hence the first term in \eqref{eqn:diff} is
  \begin{equation}\label{eqn:diff1}
    \rho_*(\omega_\rho^{-1}) = \Sym^3(\Sym^2 \mathcal A) \otimes \mathcal L^3.
  \end{equation}
  Since $\mathcal Z$ is the disjoint union $\Sigma \cup \tau$, the second term is a sum
  \[ \rho_*(\omega_\rho^{-1}|_{\mathcal Z}) = \rho_*(\omega_\rho^{-1}|_{\tau}) + \rho_*(\omega_\rho^{-1}|_{\Sigma}). \]
  Since $\tau$ is defined by the surjection $\Sym^2 \mathcal A \to \mathcal L^{-1}$, the restriction $\omega_\rho^{-1}|_\tau$ is trivial.
  To compute the restriction to $\Sigma$, we view $\Sigma$ as a divisor in $\overline{\mathcal C} = \P A^\vee$ and let $D = \omega_\rho^{-1}|_{\overline C} = \O(6) \otimes \overline \phi^* \mathcal L^3$.
  Thus, in the Grothendieck group, we have
  \begin{equation}\label{eqn:diff2}
    \rho_*(\omega_\rho^{-1}|_{\Sigma}) = \overline\phi_*\left(D\right) - \overline\phi_* \left( D(-\Sigma)\right).
  \end{equation}
  We now use \autoref{lem:push} with the section $\sigma_\infty$ to compute the terms on the right hand side of \eqref{eqn:diff2}.
  Putting this together with \eqref{eqn:diff} and \eqref{eqn:diff1}, we get
  \[ [\mathcal V] = \mathcal L^{-3}+ \mathcal L^{1}+ \mathcal L^{-1}+ \mathcal L^{-2}\]
  in the Grothendieck group.
  The result is now a simple computation using $c_1(\mathcal L) = \psi_\infty$ and $\int_{\overline \M_{0,7}} \psi_\infty^4 = 1$.
\end{proof}

Having computed the Chern classes, we find the degree of the map $\overline \M_{0,7} \dashrightarrow \M$.
\begin{proposition}\label{prop:deg2}
  The degree of the map $\overline \M_{0,7} \dashrightarrow \M$, induced by $\pi \colon \mathcal X \to \overline \M_{0,7}$ is $2 \times 72 \times 6!$.
\end{proposition}
\begin{proof}
  The map lifts to $\overline \M_{0,7} \dashrightarrow \M^\dagger$.
  Since the degree of $\M^\dagger \to \M$ is $72 \times 6!$, we must prove that the degree of $\overline \M_{0,7} \dashrightarrow \M^\dagger$ is $2$.

  Recall that $\M^\dagger$ is isomorphic to the configuration space of 6 points in $\P^2$.
  Given a general 6-tuple of points $(s_1,\dots, s_5, y)$ in $\P^2$, there is a unique conic $C$ passing through $s_1,\dots, s_5$.
  Let $x$ and $y$ be the two points of $\Polar_C(t) \cap C$.
  Then the two pre-images of $(s_1,\dots, s_5, t)$ in $\overline \M_{0,7}$ are $(C, x, s_1, \dots, s_5, y)$ and $(C,y, s_1, \dots, s_5, x)$.
  The proof is thus complete.
\end{proof}

Using \autoref{prop:goodisgood} and the computation of the Chern classes and the degree, we obtain the following relation on the undetermined coefficients of $[\Orb(X)]$ in
\eqref{eqn:chernexpansion} for a general cubic surface $X$:
\begin{align}
  \label{eq:relation2}
  625 a_{1^{4}} + 125 a_{1^{2}2} - 25a_{13} + 25a_{2^2} - 6 a_{4} = 2 \times 72 \times 6!.
\end{align}


\subsection{The third test family}
\label{sec:family-b_3}
\begin{figure}
    \centering
    \begin{tikzpicture}[thick]
  \draw [domain=0:1.4, samples=100] 
  plot ({\x^2}, {\x^3} )
  plot ({\x^2}, {-\x^3} );
  \draw (2, -3) node {$C$};
  \draw[red, fill] (1.3^2, 1.3^3) circle (0.05) node [black, left] {$t_1$};
  \draw[red, fill] (1.18^2, 1.18^3) circle (0.05) node [black, left] {$t_2$};
  \draw[red, fill] (1.0^2, 1.0^3) circle (0.05) node [black, left] {$t_3$};
  \draw[red, fill] (1.2^2, -1.2^3) circle (0.05) node [black, left] {$t_5$};
  \draw[red, fill] (0.9^2, -0.9^3) circle (0.05) node [black, left] {$t_4$};
\end{tikzpicture}

    \caption{In the family $B_3$, we blow up the points $t_i$ for $i = 1, \dots, 5$ and a fixed point $t_\infty$ at infinity (not shown) as the points $\{t_1, \dots, t_5\}$ move freely on a fixed cuspidal cubic $C$ with a flex at $t_\infty$.}
    \label{fig:FigureB3}
\end{figure}
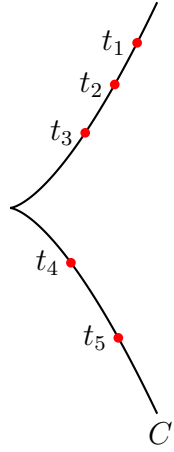
We take a cuspidal plane cubic $C \subset \P^2$ and let $t_\infty$ be its unique flex point.
We vary 5 points $\{t_1, \dots, t_5\}$ on $C$ and blow up the plane at $\{t_1, \dots, t_5, t_\infty\}$ to get a cubic surface (see \autoref{fig:FigureB3}).
Since $\Aut(C)$ is one dimensional, this gives us a 4-parameter family of cubic surfaces.

We now convert the above construction into a good family of cubic surfaces over a complete base.
We take our base to be $\widehat \M_{0,7}$, one of the compactifications of $\M_{0,7}$ constructed by Hassett \cite{has:03}.
Precisely, set
\[
  \mathbf w = \left(1-\epsilon, \epsilon, \dots, \epsilon, 1-3 \epsilon \right), \text{ where } 0 < \epsilon \ll 1,
\]
and let $\widehat \M_{0,7}$ be the space of $\mathbf w$-weighted stable pointed curves $(C, s_0, s_1, \dots, s_5, s_\infty)$.
Such curves are easy to describe:
\begin{itemize}
\item $C$ must be a smooth rational curve,
\item $s_{0}$ and $s_{\infty}$ must be distinct,
\item at most one $s_1, \dots, s_5$ can coincide with $s_{0}$,
\item at most three of the $s_1, \dots, s_5$ can coincide with $s_{\infty}$.
\end{itemize}
There are no constraints on how many of the the points $s_1, \dots, s_5$ coincide with each other.
Let $\nu \colon C \to \P^2$ be the map which is birational onto its image, whose image is a cuspidal cubic, and under which $s_0$ maps to the unique cusp point and $s_\infty$ to the unique flex point.
Let $Z \subset C$ be the subscheme corresponding to the divisor $s_1 + \dots + s_5 + s_\infty$.
Since at most one of the points $s_1, \dots, s_5$ can coincide with $s_0$, it follows that $\nu \colon Z \to \P^2$ is a closed embedding and the image is a curvilinear subscheme.
It is also easy to check that the image is cut out by cubics.
That is, $Z \subset \P^2$ is an admissible subscheme (see \autoref{def:admissible}).
We let $X_Z$ be the associated cubic surface (see \autoref{def:cubicZ}).

We now formalize the construction as a family over $\widehat \M_{0,7}$.
Let $\phi \colon \mathcal C \to \widehat \M_{0,7}$ be the universal curve with universal sections $\sigma_0, \sigma_1, \dots, \sigma_5, \sigma_\infty$.
Observe that $\phi$ is a $\P^1$-bundle.
Consider the rank 2 bundle
\begin{equation}\label{eqn:hatA}
  \mathcal A = \phi_* \O_{\mathcal C}(\sigma_\infty).
\end{equation}
As before, the evaluation maps yield an isomorphism
\begin{align*}
  \mathcal A &\xrightarrow{\sim} \sigma_\infty^*\O_{\mathcal C}\left( \sigma_\infty \right) \oplus \sigma_0^*\O_{\mathcal C}\left( \sigma_\infty \right)\\
  &= \mathcal L^{-1} \oplus \O_{\widehat \M_{0,7}},
\end{align*}
where $\mathcal L = \sigma_\infty^*(\omega_\phi)$.
We have a map $\mathcal C \to \P \mathcal A^\vee$ induced by the surjection $\phi^* \mathcal A \to \O_{\mathcal C}(\sigma_\infty)$, and since $\mathcal C \to \widehat \M_{0,7}$ is a $\P^1$-bundle, this map is an isomorphism.
Observe that the sections $\sigma_0$ and $\sigma_\infty$ correspond to the projections $\mathcal A \to \O$ and $\mathcal A \to \mathcal L^{-1}$, respectively.
Now consider the summand
\begin{equation}\label{eqn:hatS}
  \mathcal S := \O_{\widehat \M_{0,7}} \oplus \mathcal L^{-2} \oplus \mathcal L^{-3} \subset \Sym^3 \mathcal A,
\end{equation}
and denote by
\[ \nu \colon \mathcal C \to \P \mathcal S^\vee\]
the map induced by the surjection $\phi^* \mathcal S \to \O_{\mathcal C}(3\sigma_\infty)$.
Let $\rho \colon \P\mathcal S^\vee \to \widehat \M_{0,7}$ be the structure map.
Thanks to our choice of $\mathcal S$, fiber-wise the map $\nu$ maps $\mathcal C$ to cuspidal cubics with $s_0$ mapping to the cusp and $s_\infty$ to the flex.
Let $\mathcal Z \subset \mathcal C$ be the subscheme corresponding to the divisor $\sigma_1 + \dots + \sigma_\infty$.
Our family 
\[ \pi \colon \mathcal X \to \widehat \M_{0,7}\]
is the family of cubic surfaces associated to the family of admissible subschemes $\nu \colon \mathcal Z \to \P \mathcal S^\vee$.

\begin{proposition}\label{prop:goodness3}
  The family $\pi \colon \mathcal X \to \widehat \M_{0,7}$ is good.
\end{proposition}
\begin{proof}
  We show that no fiber $\mathcal X_b$ is in the closure of the orbit of any smooth cubic surface $S$ without Eckardt points.
  Let $Z \subset \P^2$ be the image of $\nu \colon Z_b \to \P \mathcal S^\vee_b$.
  By construction, $Z$ does not have a point of multiplicity 6, and since $Z$ lies on a cubic, no four distinct points of $Z$ can be collinear.
  Suppose $Z$ does not have a point of multiplicity 4.
  Then we use \autoref{prop:cubictopoints} and \autoref{prop:limit4} to conclude that $\mathcal X_b$ is not in the orbit closure of $S$.

  We now treat the cases where $Z$ has a point of multiplicity 4 or 5.
  We repeatedly use the following fact about the geometry of the cuspidal cubic $C$: no line through the flex point is tangent to $C$, except the tangent line at the flex and the line joining the flex and the cusp.

  Suppose $Z = p + q + 4r \subset C$, where $p, q, r$ are distinct.
  Since $r$ occurs with multiplicity greater than 1 in $Z$, it must be distinct from the cusp point (so $4r$ is unambiguous).
  If the tangent line $T_rZ = T_rC$ does not contain $p$ or $q$, we appeal to \autoref{prop:pq4r} and \autoref{prop:2p4r}.
  Suppose $T_rZ$ contains $p$.
  Then $T_rZ$ cannot contain $q$, and hence $p, q, r$  are not collinear.
  By the construction of $Z$, one of $p,q,r$ must be the flex point $s_\infty$.
  Using the fact about tangent lines through the flex,  we see that the only possibility is $q = s_\infty$.
  Now $Z \subset C$ is uniquely determined up to an automorphism of $(\P^2, C)$.
  It is easy to write down such a $Z$ and check directly that $\Aut(\P^2, Z)$ is finite.
  By \autoref{prop:good}, $\Aut(\mathcal X_b)$ is also finite, and $\mathcal X_b$ is not in the orbit closure of $S$.
  Alternatively, we see that the cubic surface associated to $Z$ has a unique singular point, which is of type $D_4$.
  It is known that, up to isomorphism, there are exactly two cubic surfaces with singularity $D_4$ (see \cite[Lemma~4]{bru.wal:79}).
  One has a finite automorphism group and the other one is treated by \autoref{prop:isogood}.

  Suppose $Z = p + 5q$, where $p$ and $q$ are distinct.
  Then we must have $p = s_\infty$ and $q \not \in \{s_0, s_\infty\}$.
  Let $Q$ be the unique conic containing $5q$.
  We claim that $C$ cannot contain $p$.
  If it did, then we get $5(q-p) = 0 \in \Pic(C) \cong \G_a$, forcing $p = q$.
  We also claim that the two points of $\Polar_Q(p) \cap Q$ are distinct from $q$.
  This is because $Q$ and $C$ share the tangent line at $q$; this line does not pass through $p$, but the tangent lines to $Q$ at the points of $\Polar_Q(p)$ do pass through $p$.
  Now, any $\sigma \in \Aut(\P^2, Z)$ must preserve $C$ the three points $\Polar_Q(p) \cap Q \cup \{q\}$.
  Then $\Aut(\P^2, Z)$ and hence $\Aut(\mathcal X_b)$ is finite, and $\mathcal X_b$ is not in the orbit closure of $S$.
\end{proof}

It is time for enumerative computations.
\begin{proposition}\label{prop:v3}
  Let $\mathcal V = \pi_* \omega_{\pi}^{-1}$ and let $v_i = c_i(\mathcal V)$.
  Then on $\widehat \M_{0,7}$, we have
  \[
    \begin{split}
      \int_{}v_{1}^{4} = 3436&, \quad \int_{}v_{1}^{2}v_{2} = 1076,\quad \int_{}v_{1}v_{3} = 116,\\
      &\int_{}v_2^{2} = 316,\quad \int_{}v_{4} = 0.
    \end{split}
  \]
\end{proposition}
\begin{proof}
  We begin by finding the class of $\mathcal V$ in the Grothendieck group of $\widehat \M_{0,7}$.
  The computation parallels the computation in the proof \autoref{lem:v2}, so we will be brief.
  Recall that $\phi \colon \mathcal C \to \widehat \M_{0,7}$ is the universal curve with sections $\sigma_0, \sigma_1, \dots, \sigma_\infty$.
  Setting $\mathcal L = \sigma_\infty^* \omega_\phi$, recall the bundles
  \begin{align}
    \mathcal A = \O \oplus \mathcal L^{-1} \text{ from \eqref{eqn:hatA}, and }
    \mathcal S = \O \oplus \mathcal L^{-2} \oplus \mathcal L^{-3} \text{ from \eqref{eqn:hatS}}.
  \end{align}
  Recall that $\rho \colon \P \mathcal S^\vee \to \widehat \M_{0,7}$ is the structure map and $\mathcal Z \subset \mathcal C$ the subscheme corresponding to the divisor $s_1 + \dots + s_5 + s_\infty$.
  In the Grothendieck group of $\widehat \M_{0,7}$, we have
  \begin{equation}\label{eqn:v31}
    \begin{split}
      \mathcal V &= \rho_* \omega^{-1}_\rho - \phi_* \left( \nu^* \omega^{-1}_\rho |_{\mathcal Z}\right)\\
      &= (\Sym^3 \mathcal S) \otimes \det \mathcal S^\vee - \phi_*\left(\mathcal \O\left(9\sigma_\infty\right) \otimes \phi^*\det \mathcal S^\vee|_{\mathcal Z}\right)\\
      &= (\Sym^3 \left( \O \oplus \mathcal L^{-2} \oplus \mathcal L^{-3} \right)) \otimes \mathcal L^5 - \phi_*\left(\mathcal \O\left(9\sigma_\infty\right) \otimes \phi^*\mathcal L^5\right) \\
      & \qquad +  \phi_*\left(\mathcal \O\left(9\sigma_\infty-{\mathcal Z}\right) \otimes \phi^*\mathcal L^5\right)\\
      &= \mathcal L^{-1} - \mathcal L^4 + (\mathcal L^2 + \mathcal L^3 + \mathcal L^4 + \mathcal L^5) \otimes \O\left(-\Delta_0\right).       
    \end{split}
  \end{equation}
  Here $\Delta_0 \subset \widehat \M_{0,7}$ is the divisor $\sigma_0^* \O(\sigma_1 + \dots + \sigma_5)$.
  This is the locus of marked curves where the marked point $\sigma_0$ coincides with $\sigma_i$ for some $i = 1, \dots, 5$.
  In the last step of the computation, we have used \autoref{lem:push} with the section $\sigma_0$ and that 
  \[ \sigma_0^* \O(\sigma_0) = \sigma_\infty^*\O(-\sigma_\infty) = \mathcal L,\]
  which holds since $\sigma_0$ and $\sigma_\infty$ are two disjoint sections of a $\P^1$-bundle.
  The rest is a straightforward computation using the Whitney sum formula.
\end{proof}
Let us indicate how to carry out the Chern class computation on $\widehat \M_{0,7}$.
The boundary divisors of the Hassett spaces are products of smaller Hassett spaces.
We can use this inductive structure to break down the computation, similar to how it is done on the usual $\overline \M_{0,7}$.
Alternatively, we can use the map $\zeta \colon \overline \M_{0,7} \to \widehat \M_{0,7}$ to pull back the classes to $\overline \M_{0,7}$, and do the computation there.
Let $\widehat \psi_i = c_1(\sigma_i^* \omega_\phi)$ denote the $\psi$-classes on $\widehat \M_{0,7}$ and $\psi_i$ its analogue on $\overline \M_{0,7}$.
Let $\Delta_{i,j} \subset \overline \M_{0,7}$ be the boundary divisor whose general point corresponds to the nodal union of two smooth rational curves with one component only containing the marked points indexed $i$ and $j$.
It is easy to see that
\begin{align*}
  \zeta^* \widehat \psi_0 &= \psi_0 - \sum_{i= 1}^5 [\Delta_{0,i}], \text{ and }\\
  \zeta^* [\Delta_{0}] &= \sum_{i= 1}^5 [\Delta_{0,i}],
\end{align*}
using which we can transport the entire Chern class computation to $\overline \M_{0,7}$.

We now compute the degree of the rational map from $\widehat \M_{0,7}$ to the moduli space $\M$ of cubic surfaces, induced by the family $\mathcal X \to \widehat \M_{0,7}$.
\begin{proposition}\label{prop:deg3}
  The degree of the map $\widehat \M_{0,7} \dashrightarrow \M$ is $20 \times 72 \times 6!$.
\end{proposition}
\begin{proof}
  The map lifts to $\widehat \M_{0,7} \dashrightarrow \M^\dagger$.
  The degree of this map is the answer to the following enumerative problem: given 6 general points in $\P^2$, how many cuspidal cubics contain the first five and have a flex point at the sixth?
  Fortunately for us, this question has been answered already---the answer is 20.
  The answers appears in a remarkable paper \cite{mir.des:89} of Miret and Xamb\'o, who credit it to Schubert. 
  The paper contains a staggering number of enumerative results about cuspidal cubics.
  Our answer is in the first column in the row labelled $v^2$ in the second table (titled ``Order 2'') in Section~12 of \cite{mir.des:89}.

  We briefly explain how to read off the result from the table in \cite{mir.des:89}.
  The letter $v$ denotes the map from the (7-dimensional) space of cuspidal cubics to $\P^2$,
  the symbol $X_0$ the locus of cuspidal cubics passing through a general point in $\P^2$, and the symbol $X_1$ the locus of cuspidal cubics tangent to a general line in $\P^2$.
  The $i$-th column in the row labelled $v^2$ is the intersection number $(v^* H)^2 \cdot X_0^{6-i}X_1^{i-1}$, where $H \subset \P^2$ is a general line.
  For $i = 1$, this is precisely the number of cuspidal cubics flexed at a given point and passing through 5 general points.

  Since the degree of $\M^\dagger \to \M$ is $72 \times 6!$, the degree of $\widehat \M_{0,7} \dashrightarrow \M$ is $20 \times 72 \times 6!$.
\end{proof}

Using \autoref{prop:goodisgood} and the computation of the Chern classes and the degree, we obtain the following relation on the undetermined coefficients of $[\Orb(X)]$ in
\eqref{eqn:chernexpansion} for a general cubic surface $X$:
\begin{equation}
  \label{eq:relation3}
  3436a_{1^4} + 1076a_{1^{2}2} + 116 a_{13} + 316 a_{2^{2}} = 20 \times 72 \times 6!.
\end{equation}


\subsection{The fourth test family}
\label{sec:family-b_4}
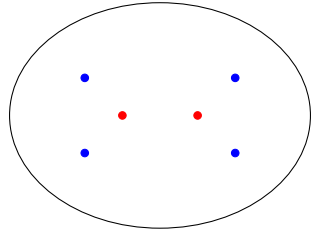
\begin{figure}
    \centering
    \begin{tikzpicture}
  \draw (0,0) ellipse (2 and 1.5);
  \begin{scope}[xshift=-1cm, yshift=-0.5cm]
  \draw[blue,fill]
  (0,1) circle (0.05)
  (2,0) circle (0.05)
  (2,1) circle (0.05)
  (0,0) circle (0.05)
  ;
\end{scope}
  \draw[red,fill]
  (0.5,0) circle (0.05)
  (-0.5,0) circle (0.05)
  ;
\end{tikzpicture}

    \caption{In the fourth test family, we blow up the plane at 4 fixed points (blue) and 2 variable points (red).}
    \label{fig:family4}
  \end{figure}
Our fourth family is perhaps the most obvious one.
We blow up 6 points on $\P^2$, four of which remain fixed, and the remaining two vary freely (see \autoref{fig:family4}).
A complication arises, however, while formalizing this construction.
When the two varying points lie on the line joining two of the fixed points, the resulting configuration is no longer admissible.
It takes some effort to resolve this issue.

Let $S$ be the blow up of $\P^2$ at 4 general points.
Then $S$ is a quintic del Pezzo. 
Let $S^{[2]}$ denote the Hilbert scheme of 2 points on $S$.
We have 10 exceptional curves $L_1, \dots, L_{10}$ on $S$.
Consider the planes $L^{[2]}_i \subset S^{[2]}$.
Observe that a length 2 subscheme $Z \subset S$ is not admissible precisely when $Z$ is contained in $L_i$ for some $i$.
For $i \neq j$, the lines $L_i$ and $L_j$ are either disjoint or intersect in a single (reduced) point.
Therefore, $L^{[2]}_i$ and  $L^{[2]}_j$ are disjoint subsets of $S^{[2]}$.
Let $\Lambda$ be the union of all $L^{[2]}_i$.
The crux of this
section is to show that the blow up $\Bl_{\Lambda}S^{[2]}$ hosts a
good family of cubic surfaces.

Set
\begin{equation}
  \label{eq:S2tilde}
  \widetilde {S^{[2]}} := \Bl_{\Lambda}S^{[2]},
\end{equation}
with $\beta \colon \widetilde {S^{[2]}} \to S^{[2]}$ being the blow-down map.
For $i=1, \dots, 10$, let $E_{i} = \beta^{-1}(L^{[2]}_{i})$ denote the
components of the of the exceptional divisor of the blow-up.
Consider
\[ F_i: = E_i \times L_i \subset \widetilde {S^{[2]}} \times S,\]
and let $F$ be the (disjoint) union $F = \bigcup_i F_i$.
Set
\begin{equation}
  \label{eq:Xtilde}
  \mathcal{Y} := \Bl_{F}\left( \widetilde{S^{[2]}} \times S \right),
\end{equation}
with $\eta \colon \mathcal Y \to  \widetilde{S^{[2]}} \times S$ being the blow-down map.
Consider the map
\[ \widetilde \pi \colon \mathcal Y \to \widetilde {S^{[2]}}\]
obtained by composing $\eta$ with the projection on to the first factor.
It is easy to check that $\widetilde \pi$ is a flat family of surfaces, isomorphic to the constant family with fiber $S$ over the complement of $E$ in $\widetilde{S^{[2]}}$.

Let $\mathcal Z \subset S^{[2]} \times S$ be the universal closed subscheme of length 2.
Let $\widetilde {\mathcal Z} \subset \widetilde{S^{[2]}} \times S$ be the fiber product
\[ \widetilde {\mathcal Z} = \mathcal Z \times_{S^{[2]}} \widetilde {S^{[2]}}.\]
The next lemma is critical to our construction.
\begin{lemma}
  The closed embedding
  $i: \widetilde{\mathcal{Z}} \hookrightarrow \widetilde{S^{[2]}}
  \times S$ lifts to a closed embedding
  $j: \widetilde{\mathcal{Z}} \hookrightarrow  \mathcal{Y}$.
\end{lemma}
\begin{proof}
  The subscheme $(L^{[2]}_i \times S) \cap \mathcal Z$ is contained in $(L^{[2]}_i \times L_i)$.
  Hence, we have an equality of schemes
  \[
    (L^{[2]}_i \times L_i) \cap \mathcal Z = (L^{[2]}_i \times S) \cap \mathcal Z.
  \]
  By pulling back both sides to $\widetilde{S^{[2]}} \times S$, we get
  \begin{equation}
    \label{eq:Cartier}
    F_i \cap \widetilde{\mathcal{Z}} = (E_i \times S) \cap \widetilde{\mathcal{Z}}.
  \end{equation}
  In particular, $F_i \cap \widetilde {\mathcal Z}$ is a Cartier divisor.
  By the universal property of the blow up, we get a lift $j \colon \widetilde{\mathcal Z} \to \mathcal Y$.
  Since $i = \beta \circ j$ is a closed embedding, so is $j$.
\end{proof}

The following diagram summarizes the situation so far:
\begin{center}
\begin{tikzcd} & \mathcal{Y} \arrow[d, "\eta"] & & \\ \td{\mathcal{Z}} \arrow[r,
  hook, "i"] \arrow[ur, hook, "j"] & \td{S^{[2]}} \times S \arrow[r] \arrow[d] & S^{[2]}\times S \arrow[d] \\ & \td{S^{[2]}} \arrow[r, "\beta"] &
  S^{[2]}.
\end{tikzcd}
\end{center}

We let $\mathcal{F} \subset \mathcal{Y}$ denote the exceptional
divisor of $\eta$. It has $10$ disjoint components, corresponding to
each $F_{i}$, which in turn correspond to the ten lines $L_{i}$.
To ease notation, we suppress the map $j \colon \widetilde{\mathcal Z} \to \mathcal Y$, and write $\widetilde{\mathcal Z} \subset \mathcal Y$.

Let us study the family $\widetilde \pi \colon \mathcal Y \to \widetilde{S^{[2]}}$.
Its fiber over a point in the complement of $E$ is just $S$.
The following proposition describes the fiber of $\mathcal Y$ over $e \in E_i$, along with the fiber of $\widetilde{\mathcal Z}$ in it; see \autoref{fig:familyY} for a picture.
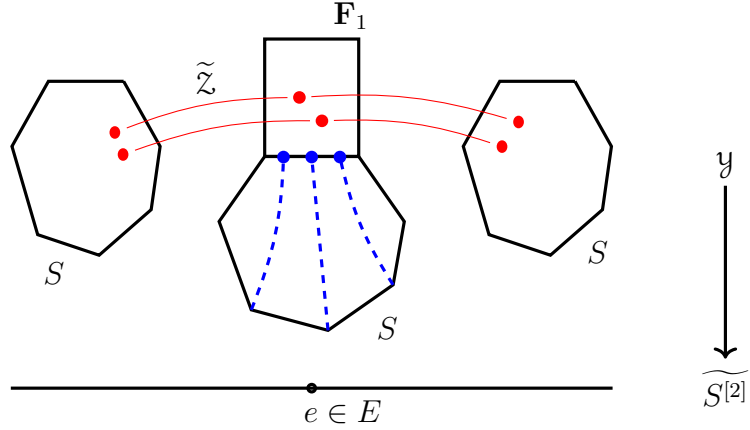
\begin{figure}
  \centering
    \begin{tikzpicture}[very thick]
  \draw (-4,0) -- (4,0);
  \draw (5.5, 3) node (Y) {$\mathcal Y$};
  \draw (5.5,0) node (S){$\widetilde {S^{[2]}}$};
  \draw (Y) edge[->] (S);
  \draw (0,0) circle (0.05) node [below] {$\qquad e \in E$};
  
  \begin{scope}[xshift=-3cm, yshift=3cm, yscale=1.25]
    \draw (60:1) -- (120:1) -- (170:1) -- (230:1) -- (280:1) -- (330:1) -- (10:1) -- (60:1);
    \draw[fill, red]
    (10:0.5) node (a1) {} circle (0.05)
    (40:0.5) node (a2) {} circle (0.05);
    \draw (250:1.25) node {$S$};
  \end{scope}

  \begin{scope}[xshift=3cm, yshift=3cm, yscale=1.25]
    \draw (60:1) -- (120:1) -- (170:1) -- (230:1) -- (280:1) -- (330:1) -- (10:1) -- (60:1);
    \draw[fill, red]
    (160:0.5) node (b1) {} circle (0.05)
    (120:0.5) node (b2) {} circle (0.05);
    \draw (310:1.25) node {$S$};
  \end{scope}

  \begin{scope}[yshift=2cm, scale=1.25]
    \draw (60:1) -- (120:1) -- (170:1) -- (230:1) -- (280:1) -- (330:1) -- (10:1) -- (60:1);
    \draw (60:1) -- (120:1) -- ++(0,1.25) -- ++(1,0) -- cycle;
    \draw[fill, red]
    (85:1.25) node (z1) {} circle (0.05)
    (95:1.50) node (z2) {} circle (0.05);
    \draw[blue,fill] (230:1) edge [dashed, bend right=10] (-0.3,0.86)
    (-0.3,0.86) circle (0.05);
    \draw[blue,fill] (280:1) edge [dashed] (0,0.86)
    (0,0.86) circle (0.05);
    \draw[blue,fill] (330:1) edge [dashed, bend left=10] (0.3,0.86)
    (0.3,0.86) circle (0.05);
    \draw (310:1.25) node {$S$};
    \draw (80:2.4) node {$\F_1$};
  \end{scope}
  \draw[thin, draw=red, bend left=10] (a1) edge (z1) (z1) edge (b1);
  \draw[thin, draw=red, bend left=10] (a2) edge node[above] {$\widetilde{\mathcal Z}$} (z2)
  (z2) edge (b2);
\end{tikzpicture}

    \caption{The family of surfaces $\mathcal Y \to \widetilde {S^{[2]}}$ and the subscheme $\widetilde {\mathcal Z} \subset \mathcal Y$.
      The dashed lines on the central fiber represent the three exceptional curves of $S$ intersecting the double curve.
    }
    \label{fig:familyY}
\end{figure}

\begin{proposition}
  \label{prop:geometry}
  Let $e \in E_{i}$ be any point, and let
  $\mathcal{Y}_{e} = \widetilde \pi^{-1}(e)$.
  \begin{enumerate}
  \item $\mathcal{Y}_{e}$ is the union of two surfaces $S$ and $\mathcal F_e$, where $\mathcal F_e$ is a copy of the Hirzebruch surface $\F_1$.
    They meet transversely along $L_{i} \subset S$ and a smooth curve of self-intersection $1$ in $\F_1$.
  \item The subscheme $\widetilde{\mathcal Z}_e \subset \mathcal Y_e$ lies in $\mathcal F_e$ and is disjoint from $S$.
    It maps isomorphically onto its image in $L_{i}$ under $\eta$.
  \end{enumerate}
\end{proposition}

\begin{proof}
  We study the blowup $\mathcal{Y} = \Bl_{F} \left( \td{S^{[2]}} \times S
  \right)$, beginning with the normal bundle $N_{F/ \td{S^{[2]}} \times S }$.
  We focus on the component $F_{i} = E_{i} \times L_{i}$.
  Denoting by ${\pi}_1 \colon F_i \to E_i$ and ${\pi}_2 \colon F_i \to L_i$ the two projections, we have
  \[N_{F_{i}/ \td{S^{[2]}} \times S } = {\pi}_{1}^{*}
    N_{E_{i}/\td{S^{[2]}}} \oplus {\pi}_{2}^{*}N_{L_{i}/S}.\]
  The component of $\mathcal{F}$ lying above $F_{i}$ is the
  projectivization $\P N_{F_{i}/ \td{S^{[2]}} \times S }$.
  Restricting the projectivization to $\{e\} \times L_{i} \simeq L_i$ yields
  $\P (\O_{L_{i}} \oplus \O_{L_{i}}(-1))$.
  Thus we get  $\mathcal{F}_{e} \simeq \F_{1}$.
  Since $(\{e\} \times S) \cap F_{i} = \{e\} \times L_{i}$ is a Cartier
  divisor on $\{e\} \times S$, the proper transform of $\{e\} \times S$ in
  $\mathcal{Y}$ is again a copy of $S$.
  It meets
  $\mathcal{F}_{e} = \P (\O_{L_{i}} \oplus \O_{L_{i}}(-1)) \simeq
  \F_{1}$ along the section corresponding to the summand
  $N_{L_{i}/S} = \O_{L_{i}}(-1)$.
  Altogether, we get the description of $\mathcal{Y}_{e}$ provided in (1).

  Since the embedding $i \colon \widetilde {\mathcal Z} \to \widetilde{S^{[2]}} \times S$ is the composite
  \[\widetilde {\mathcal Z} \xrightarrow{j} \mathcal Y \xrightarrow{\eta} \widetilde{S^{[2]}} \times S,\]
  and $\mathcal F = \eta^{-1}(F)$, we have
  \[ \widetilde {\mathcal Z} \cap j^{-1}\left(\mathcal F\right) = \widetilde {\mathcal Z} \cap i^{-1}(F).\]
  We have already seen in \eqref{eq:Cartier} that
  \[F \cap \td{\mathcal{Z}} = (E \times S)\cap \td{\mathcal{Z}}.\]
  Restricting the first coordinate of $E \times S$ to $e$ implies
  $\td{\mathcal{Z}_{e}} \subset \mathcal{F}_{e}$.

  To see that $\widetilde Z_e$ is disjoint from $S$, consider
the intersection product
$[\td{\mathcal{Z}}] \cdot [\mathcal{Y}_{e}]$.  On the one hand, since
$\td{\mathcal{Z}} \to \td{S^{[2]}} \times S$ is finite,
flat, and has degree $2$, this intersection product is $2$.  On the
other hand, by part (1), it is equal to
$[\td{\mathcal{Z}}] \cdot ([\mathcal{F}_{e}]+ [S])$.
Since we have
$\td{\mathcal{Z}}_{e} \subset \mathcal{F}_{e}$, we get
$[\td{\mathcal{Z}}] \cdot [S] = 0$, and hence $\td{\mathcal Z}$ must be disjoint from $S$.

The last assertion follows immediately from the fact that $\widetilde Z \subset \mathcal Y$ maps isomorphically to its image in $\widetilde{S^{[2]}} \times S$.
\end{proof}

The following proposition provides one more necessary detail about the position of $\widetilde {\mathcal Z}_e$ in the $\F_1$ component of $\mathcal Y_e$.
\begin{proposition}
  \label{prop:nodirectrix}
  Maintain the notation of \autoref{prop:geometry}.
  \begin{enumerate}
  \item The length two
    subscheme $\td{\mathcal{Z}}_{e} \subset \F_{1}$ is not contained in the directrix of $\F_{1}$.
  \item Set $\tau_i = T_i \cap \F_1$.
    The subscheme $\td{\mathcal Z}_e \cup \tau_i$ is an admissible subscheme of $\F_1$.
  \end{enumerate}
\end{proposition}
\begin{proof}
  The proof follows easily from unravelling the map
  \[ j \colon \widetilde {\mathcal Z}_e \to \F_1,\]
  which we now do.
  Set $Z = \widetilde{\mathcal Z}_e$ and let $N_e$ be the 1-dimensional vector space
  \[ N_e = N_{E_i/\widetilde{S^{[2]}}} \big|_e. \]
  From the proof of \autoref{prop:geometry}, we write the $\F_1$ component more canonically as
  \[
    \F_1 = \P\left(N_e \otimes \O_{L_i} \oplus N_{L_i/S}\right).
  \]
  The point $e \in E_i$ gives a (non-zero) map
  \[
    d \beta \colon N_e \to N_{L^{[2]}_i / S^{[2]}}|_{\beta(e)}.
  \]
  The point $\beta(e) \in L^{[2]}_i \subset S^{[2]}$ corresponds to the length 2 subscheme $Z \subset L_i \subset S$.
  It is easy to see that we have a canonical identification
  \[
    N_{L^{[2]}_i / S^{[2]}}\big |_{\beta(e)} = H^0\left(Z, N_{L_i/S}|_Z\right).
  \]
  We thus get a map
  \[
    d \beta \colon N_e \to H^0\left(Z, N_{L_i/S}|_Z\right).
  \]
  We leave it to the reader to verify that the map $j \colon Z \to \F_1$ is induced by the map
  \begin{equation}\label{eqn:emap}
    N_e \otimes \O_Z \to N_e \otimes \O_Z \oplus N_{L_i/S}|_Z
  \end{equation}
  given by $(\id, \epsilon \circ d\beta)$, where
  \[
    \epsilon \colon H^0\left(Z, N_{L_i/S}|_Z\right) \otimes \O_Z \to N_{L_i/S}|_Z
  \]
  is the evaluation map.
  Since $\epsilon$ is an isomorphism and $d \beta$ is non-zero, the second coordinate of \eqref{eqn:emap} is non-zero.
  This non-vanishing equivalent to the proposition.
  (Incidentally, note that the first coordinate of \eqref{eqn:emap} is invertible.
  This is equivalent to the fact that $Z \subset \F_1$ is disjoint from the curve where $\F_1$ meets the other component $S$.
  In the proof of \autoref{prop:geometry}, we concluded this fact using intersection numbers.)

  The assertion about admissibility is straightforward.
\end{proof}

We are now in a position to describe the family of cubic surfaces.
Consider the family
\[ \widetilde \pi \colon \mathcal Y \to \widetilde{S^{[2]}},\]
along with $\widetilde {\mathcal Z} \subset \mathcal Y$.
Set $U = \widetilde{S^{[2]}} \setminus E$.
Over $U$, the $\widetilde{\mathcal Z}$ is an admissible subscheme of $\mathcal Y$, and here our family will be the associated cubic surface as in \autoref{def:cubicZ}.
We now describe our cubic surface over a point $e \in E_i$.
In the $S$ component of $\mathcal Y_e$, we have three exceptional curves that meet the double curve $L_i \subset S$.
They intersect $L_i$ in three distinct points.
Let $T_i$ be the union of these three exceptional curves (in \autoref{fig:familyY}, these are indicated by the dashed lines).
On the $\F_1$ component, we have a distinguished length 5 subscheme: the union of the length two subscheme $\widetilde{\mathcal Z}_e$ and the 3 points $\tau_i := T_i \cap L_i$.
This subscheme of length 5 is admissible, and our cubic surface over $e$ will be its associated cubic surface.

We construct the family that we described fiber-wise above using the linear series associated to the line bundle
\begin{equation}\label{eqn:l1}
  \mathcal L = \omega_{\widetilde \pi}^{-1} \otimes \O_{\mathcal Y}(-\mathcal F).
\end{equation}
Since $\mathcal Y$ is the blow up of $\widetilde{S^{[2]}} \times S$ with exceptional divisor $\mathcal F$, we can also write this line bundle as
\begin{equation}\label{eqn:l2}
  \mathcal L = \omega_{S}^{-1} \otimes \O_{\mathcal Y}(-2\mathcal F).
\end{equation}
The following proposition describes $\mathcal L$ on the degenerate fibers of $\widetilde \pi$.
\begin{proposition}
  \label{prop:LYe}
  Let $e \in E_{i}$ and $\mathcal{Y}_{e} = S \cup_{L_{i}} \F_{1}$ as in \autoref{prop:geometry}.
  Let $T_i \subset S$ be the union of the three exceptional curves meeting $L_i$ and set $\tau_i = L_i \cap T_i$.
  Then
  \begin{enumerate}
  \item $\mathcal{L}|_{S} \simeq \O_{S}(T_{i})$. 
  \item $\mathcal{L}|_{\F_{1}} \simeq \omega_{\F_1}^{-1}$.
  \item The restriction map
  \begin{align}
    \label{eq:restrL}
    H^{0}(\mathcal{Y}_{e}, \mathcal{L}|_{\mathcal{Y}_{e}}) \to H^{0}(\F_{1}, \omega_{\F_1}^{-1})
  \end{align}
  is an isomorphism onto the $6$ dimensional vector space
  $H^{0}\left(\F_{1}, \omega_{\F_1}^{-1}\otimes \mathcal{I}_{\tau_{i}}\right).$
\item The base locus of the complete linear system $|\mathcal L|$ on $\mathcal Y_e$ is $T_i$.
  \end{enumerate}
\end{proposition}
\begin{proof}
  We have $\omega_S^{-1} = \O(T_i + 2L_i)$.
  This and \eqref{eqn:l2} gives $\mathcal L|_S = \O_S(T_i)$.

  Let $\mathcal S \subset \mathcal Y$ be the proper transform of $E \times S \subset \widetilde{S^{[2]}} \times S$, so that the preimage in $\mathcal Y$ of $E \times S$ is the sum $\mathcal F + \mathcal S$  .
  Then we have
  \[\O(\mathcal F)|_{\F_1} = \O(-\mathcal S)|_{\F_1}.\]
  Since $\omega_S^{-1} \cdot L_i = 1$, the pull-back of $\omega_S^{-1}$ to $\F_1$ is the class of a fiber.
  Since $\omega^{-1}_{\F_1}$ is the sum of a fiber class and 2 times the class of a section of self-intersection 1, we get $\mathcal L|_{\F_1} = \omega_{\F_1}^{-1}$.

  The statements about the dimension and the base locus of $|\mathcal L|$ on $\mathcal Y_e$ are straightforward.
  The only key point is that $\O_S(T_i)$ is one-dimensional, and the vanishing locus of its non-zero sections is $T_i$.
\end{proof}

Let
\[
\mathcal W := \widetilde \pi_* \left( \mathcal L \otimes \mathcal I_{\widetilde{\mathcal Z}} \right).
\]
It follows from Grauert's theorem that $\mathcal W$ is a vector bundle of rank 4.
Consider the rational map
\[ \kappa \colon \mathcal Y \dashrightarrow \P \mathcal W^\vee\]
induced by the evaluation
\[ \widetilde \pi^* \mathcal W \to \mathcal L.\]
It is easy to check that for every $e \in \widetilde{S^{[2]}}$, the image of $\kappa(\mathcal Y_e)$ is contained in a unique cubic surface in $\P \mathcal W^\vee_e$.
Let $\mathcal X \subset \P \mathcal W^\vee$ be the resulting family of cubic surfaces.
More precisely, consider the map
\[ \Sym^3 \mathcal W \to \pi_*(\mathcal L^3)\]
and verify (by looking at the fibers) that both sides are vector bundles and the map is surjective with kernel of rank 1.
The rank 1 kernel defines the cubic $\mathcal X \subset \P \mathcal W^\vee$.
Let $\pi \colon \mathcal X \to \widetilde{S^{[2]}}$ be the obvious map.

The next proposition identifies the anti-canonical bundle of $\pi$.
\begin{proposition}
  We have $\omega_\pi^{-1} = \O(1) \otimes \pi^*\O(E)$.
\end{proposition}
\begin{proof}
  Let $\widehat {\mathcal Y}$ be the blow-up of $\mathcal Y$ along $\widetilde{\mathcal Z}$ and $\widehat \pi \colon \widehat{\mathcal Y} \to \widetilde{S^{[2]}}$ the resulting map.
  Then it is easy to check that the rational map $\widehat\kappa \colon \widehat {\mathcal Y} \dashrightarrow \mathcal X$ (induced by $\kappa$) extends to a regular map away from the union of the curves $T_i$ over the points of $E_i$.
  Recall that $\mathcal S \subset \mathcal Y$ is the proper transform of $E \times S$, so that $\widetilde \pi^* \O(E) = \mathcal F + \mathcal S$.
  The map
  \[
    \widehat \kappa \colon \widehat {\mathcal Y} \setminus \mathcal S \dashrightarrow \mathcal X
  \]
  is an isomorphism away from a locus of codimension at least 2 on both sides.
  If $\widehat D \subset \widehat {\mathcal Y}$ is the exceptional divisor of the blow-up, then we have
  \begin{align*}
    \widehat \kappa ^* \O(1) &= \mathcal L(-\widehat D) \\
                             &= \omega^{-1}_{\widetilde \pi} \otimes \O(-\mathcal F) \otimes \O_{\widehat Y}(-\widehat D) \\
                             &= \omega_{\widehat \pi}^{-1} \otimes \O(-\mathcal F).
  \end{align*}
  We see that the line bundles $\widehat \kappa^* \O(1) \otimes \widehat \pi^* \O(E)$ and $\omega_{\widetilde \pi}^{-1}$ are isomorphic on $\widehat {\mathcal Y} \setminus \mathcal S$.
  Hence, they are isomorphic away from a locus of codimension at least 2 on $\mathcal X$.
  Since $\mathcal X$ is a hypersurface in a smooth variety, and it is smooth in codimension 1, it is normal.
  It follows that $\O(1) \otimes \pi^* \O(E)$ and $\omega_\pi^{-1}$ are isomorphic on $\mathcal X$.
\end{proof}

\begin{proposition}
  \label{prop:good4} The family $\pi: \mathcal{X} \to \td{S^{[2]}}$ is
  good.
\end{proposition}
\begin{proof}
  Over $e \not \in E$, the cubic $\mathcal X_e$ is the one associated to $\mathcal Z_e \subset S$, and over $e \in E$, say $e \in E_i$, the cubic $\mathcal X_e$ is the one associated to $\widetilde {\mathcal Z}_e \cup \tau_i \subset \F_1$.
  In both cases, it is easy to verify using \autoref{prop:good} that $\mathcal X_e$ has a finite automorphism group.
\end{proof}

Having constructed the family $\pi \colon \mathcal X \to \widetilde{S^{[2]}}$, we take up the task of calculating the Chern classes of
\[ \mathcal V = \pi_* \left(\omega_\pi^{-1}\right) = \mathcal W \otimes \O_{\widetilde{S^{[2]}}}(E).\]
Let $\mathcal U = \left(\omega_S^{-1}\right)^{[2]}$ be the rank 2 tautological bundle on $S^{[2]}$, defined by
\[ \mathcal U = {\pi_1}_*\pi_2^* \left(\omega_S^{-1}\right),\]
where $\pi_i$ are the two projections on $S^{[2]} \times S$.
\begin{proposition}\label{prop:v4groth}
  In the Grothendieck group of $\widetilde {S^{[2]}}$, we have
  \[
    \mathcal V =  \O(E) + \O^3 + \O(-E)^2 - \mathcal U \otimes \O(-E),
  \]
  where $\O$ denotes $\O_{\widetilde{S^[2]}}$.
\end{proposition}
\begin{proof}
  We compute the class of $\mathcal W$ and obtain the class of $\mathcal V$ by tensoring by $\O(E)$.
  Recall the family $\widetilde \pi \colon \mathcal Y \to \widetilde{S^{[2]}}$ and the line bundle $\mathcal L$ on it defined in \eqref{eqn:l1}.
  Set
  \[ \mathcal A := \widetilde{\pi}_* \mathcal L \text{ and } \mathcal B := \widetilde{\pi}_* \left( \mathcal L|_{\widetilde{\mathcal Z}} \right).\]
  Then $\mathcal A$ and $\mathcal B$ are vector bundles of rank 6 and 2, respectively.
  By the definition of $\mathcal W$, we have the exact sequence
  \[ 0 \to \mathcal B \to \mathcal A \to \mathcal W \to 0,\]
  and hence $\mathcal W = \mathcal A - \mathcal B$ in the Grothendieck group.
  
  We first compute $\mathcal B$.
  By the push-pull formula and the description \eqref{eqn:l2} of $\mathcal L$, we get
  \begin{equation}\label{eqn:B}
    \mathcal B = \mathcal U \otimes \O_{\widetilde{S^{[2]}}}(-2E).
  \end{equation}
  
  We now compute $\mathcal A$.
  Recall that $\eta \colon \mathcal Y \to \widetilde{S^{[2]}} \times S$ is the blow-up at $F = \bigsqcup_i E_i \times L_i$.
  Using \eqref{eqn:l2} again, we get
  \[
    \eta_* \mathcal L = \omega_S^{-1} \otimes \mathcal I^2_{F}.
  \]
  Using the pair of sequences
  \begin{equation}
    \label{eq:seqA}
    \begin{split}
      &0 \to \mathcal{I}^2_{F} \to \O_{\td{S^{[2]}} \times S} \to \O_{2F} \to 0, \text{ and}\\
      &0 \to \mathcal{I}_{F}/\mathcal{I}_{F}^{2} \to \O_{2F} \to \O_{F} \to 0,
    \end{split}
  \end{equation}
  and tensoring by $\omega_S^{-1}$ gives (in the Grothendieck group)
  \begin{equation}\label{eqn:el}
    \eta_* \mathcal L = \omega_S^{-1} \otimes \left(\O_{\widetilde{S^{[2]}} \times S} - \O_F - \mathcal{I}_F / \mathcal{I}^2_F\right).
  \end{equation}
  On the component $F_i$ of $F$, the conormal bundle $\mathcal{I}_F / \mathcal{I}^2_F$ splits as a direct sum
  \[ \pi_1^*\O_{E_i}(-E_i) \oplus \pi_2^*\O_{L_i}(1),\]
  where $\pi_i \colon F_i \to E_i \times L_i$ are the two projections.
  After substituting in \eqref{eqn:el} and pushing forward to $\widetilde{S^{[2]}}$, we get
  \begin{equation}\label{eqn:A}
    \begin{split}
      \mathcal A &= \O^6- \O_E^{5}- \O_E(-E)^{2}\\
      &= \O + \O(-E)^3 + \O(-2E)^2.
    \end{split}
  \end{equation}
  Subtracting $\eqref{eqn:B}$ from $\eqref{eqn:A}$ and tensoring by $\O(E)$ gives the result.
\end{proof}

Having computed the class of $\mathcal V$, what remains is a computation in the Chow ring of $\widetilde{S^{[2]}}$.
In the computation, we suppress pull-back symbols to ease notation.
We have the following simplification.
\begin{proposition}
    \label{lemma:uE}
  Let $u_i = c_i(\mathcal U)$ for $i = 1, 2$.
  Then  $u_{1}\cdot E = u_{2} \cdot E = 0$.
\end{proposition}
\begin{proof}
  We show that $u_i \in A^i(S^{[2]})$ is supported on cycles disjoint from $\Lambda = \bigsqcup_i L^{[2]}_i$.

  For $u_2$, fix a general $C \in |\omega_S^{-1}|$.
  Then $C$ intersects each exceptional curve $L_i$ transversely in a single point.
  Then $u_2$ is the class of $C^{[2]} \subset S^{[2]}$, which is clearly disjoint from $\Lambda$.
  
  For $u_1$, fix a general pencil of curves $\{C_t \mid t \in \P^1\} \subset |\omega_S^{-1}|$.
  Then for every $t \in \P^1$, the curve $C_t$ intersects each exceptional curve $L_i$ transversely in a single point.
  Then $u_{1}$ is the class of
  \[\{[Z] \in S^{[2]} \mid Z \subset C_{t} \text{ for some $t\in \P^{1}$}\},\]
  which is clearly disjoint from $\Lambda$.
\end{proof}

We now compute the non-trivial degree 4 intersection numbers.
\begin{proposition}
  \label{prop:monomialsU}
  On $\widetilde{S^{[2]}}$, we have
  \[
    \begin{split}
      \int u_{1}^{4} = 36, \qquad   \int u_{1}^{2}u_{2} = 15, \qquad 
      \int  u_{2}^{2} = 10, \qquad    \int  E^{4} = -30.
    \end{split}
\]
\end{proposition}
\begin{proof}
  We know explicit cycles on $S^{[2]}$ that represent $u_1$ and $u_2$ (see the proof of \autoref{lemma:uE}).
  Using these, we convert the intersection numbers into enumerative problems, which we solve.

  The number $\int u_1^4$ is the answer to the following problem.
  Choose $4$ general anti-canonical pencils on $S$, and let $f \colon S \dashrightarrow (\P^1)^4$ be the rational map induced by them.
  Then $u_1^4$ is the number of double points of the map $f$ (pairs of points in $S$ that have the same image under $f$).
  We find that the number is 36 by applying the double point formula \cite[Theorem~2]{ful:78} to a resolution of $f$.

  The number $\int u_{1}^{2}u_{2}$ is the answer to the following problem.
  Fix a general anti-canonical $C \subset S$ and two general anti-canonical pencils.
  Then $\int u_1^2u_2$ is the number of double points of the map $C \to \P^1 \times \P^1$ induced by the two pencils.
  Using the genus of $C$, which is 1, and the arithmetic genus of $f(C)$, which is 16, we see that the number is 15.

  The number $\int u_2^2$ is the number of pairs of points common to two general anti-canonical curves $C_1$ and $C_2$.
  Since there are $5$ points in $C_1 \cap C_2$, the number of pairs is ${5 \choose 2} = 10$.

  Finally, to compute $\int E^4$, observe that $E$ is the disjoint union of $E_i$ for $i = 1, \dots, 10$.
  Each $E_i$ is the $\P^1$-bundle over $L_i$ defined by
  \[ E_i = \P N_{L_i^{[2]}/S^{[2]}},\]
  and the restriction of $E_i$ to $E_i$ is $\O(-1)$.
  We have $L_i^{[2]} \cong \P^2$ and
  \[ N_{L_i^{[2]}/S^{[2]}} = N_{L_i/S}^{[2]} \cong \O_{\P^2}(-1) \oplus \O_{\P^2}(-1).\]
  Now it is a straightforward to check that $E_i^4 = -3$, and hence $E^2 = -30$.
\end{proof}
\begin{proposition}\label{prop:4deg4}
  Let $\mathcal V = \pi_*\left(  \omega_{\pi}^{-1} \right)$ be the anti-canonical section bundle of $\pi \colon \mathcal X \to \widetilde{S^{[2]}}$,
  and let $v_i = c_i(\mathcal V)$.
  Then we have
  \[
    \int_{B_{4}} v_{1}^{4} = 6, \qquad
    \int_{B_{4}} v_{1}^{2}v_{2} = 21, \qquad
      \int_{B_{4}} v_{1}v_{3} = 6, \qquad
      \int_{B_{4}} v_{2}^{2} = 16, \qquad
      \int_{B_{4}} v_{4} = 1.
  \]
\end{proposition}
\begin{proof}
  Follows from the expression for $\mathcal V$ in the Grothendieck group we found in \autoref{prop:v4groth} and the Chern class computations we did in \autoref{prop:monomialsU} and \autoref{prop:4deg4}.
\end{proof}

Let $\mu \colon \widetilde{S^{[2]}} \dashrightarrow \M$ be the map to the moduli space of cubic surfaces induced by the family $\pi \colon \mathcal X \to \widetilde{S^{[2]}}$.
\begin{proposition}
  The degree of $\mu \colon \widetilde{S^{[2]}} \dashrightarrow \M$ is $36 \times 6!$.
\end{proposition}
\begin{proof}
  Recall that the moduli space of marked cubic surfaces $\M^\dagger$ is the configuration space of 6 ordered points in $\P^2$.
  The map $\mu$ clearly lifts to a birational map
  \[ \widetilde S^{[2]} \dashrightarrow \M^\dagger / S_2,\]
  where the $S_2$ permutes the last 2 points of the configuration.
  Since the degree of $\M^\dagger \to \M$ is  $72 \times 6!$, the degree of $\mu$ is half of that.
\end{proof}

Using \autoref{prop:goodisgood} and the computation of the Chern classes and the degree, we obtain the following relation on the undetermined coefficients of $[\Orb(X)]$ in
\eqref{eqn:chernexpansion} for a general cubic surface $X$:
\begin{align}
  \label{eq:relation4}
  6 a_{1^{4}} + 21 a_{1^{2}2} + 6a_{13} + 16a_{2^2}+  a_{4} = 36 \times 6!.
\end{align}

\subsection{Isotrivial families}
In this section, we continue the theme of constructing good families, but with a twist.
Let $X \subset \P^3$ be a cubic surface with an automorphism group $G$.
We then get a family of cubic surfaces
\[ \pi \colon [X/G] \to BG,\]
and hence a map from $BG$ to the moduli stack of cubic surfaces
\[ \mu \colon BG \to \mathscr M.\]
Suppose we know that the family $\pi$ is good.
In this case, this simply means that $X$ is not in the closure of a general cubic surface $S$.
Then, we get
\[ \mu^* [\Orb(S)] = 0,\]
which in turn gives a linear relation among the coefficients of the expression for $[\Orb(S)]$ pulled back to $A^4(BG)$.
To get a useful relation, the group $A^4(BG)$ must be rich.
In particular, we must take $G$ to be infinite; otherwise, $A^4(BG)$ is torsion, and we only get a congruence relation.

If we wish to obtain a family over a schematic base instead of $BG$, we can easily do so.
We consider an arbitrary scheme $B$ with the free action of $G$ such that the quotient $B/G$ is a scheme, and take the family to be
\[ \pi \colon (X \times B)/G \to B/G,\]
where $G$ acts on $X \times B$ diagonally.
In particular, for $G = \G_m$, the only group we use, we can take $B$ to be an arbitrary $\G_m$-torsor (line bundle minus the zero section) over a scheme.
See \autoref{sec:schematicexample} for an example.

To implement the strategy outlined above, we need to prove that the cubic surface $X$ does not lie in the closure of the orbit of a general cubic surface.
Following \autoref{def:goodfamily}, say that $X$ is \emph{good} if it has this property.
\autoref{prop:isogood} proves that the following cubic surfaces are good (the parenthesis describes the singularities):
\begin{enumerate}
\item $x_0x_1x_3 = x_2^3$ \quad $(3A_2)$
\item $x_3(x_0x_2-x_1^2) = x_0x_1^2$ \quad $(A_3 + 2A_1)$
\item $x_3(x_0x_2-x_1^2) = x_0^2x_1$ \quad $(A_4 + A_1)$
\item $x_3x_0^2 = x_1^3 + x_2^3$ \quad $(D_4)$
\end{enumerate}
For the reader's amusement, we include real pictures of the surfaces above in \autoref{fig:surfaces}.
\begin{figure}
  \begin{subfigure}{.24\textwidth}
    \centering
    \includegraphics[width=\textwidth]{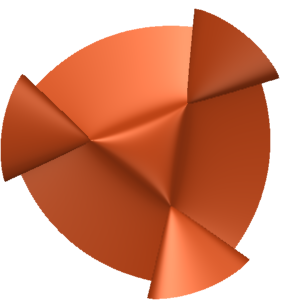}
    \caption{$3A_2$}
  \end{subfigure}
  \begin{subfigure}{.24\textwidth}
    \centering
    \includegraphics[width=\textwidth]{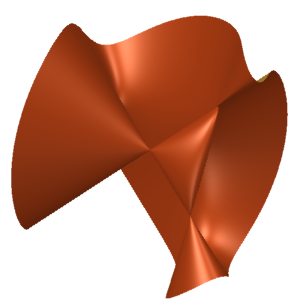}
    \caption{$A_3+2A_1$}
  \end{subfigure}
  \begin{subfigure}{.24\textwidth}
    \centering
    \includegraphics[width=\textwidth]{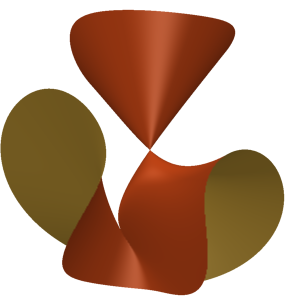}
    \caption{$A_4+A_1$}
  \end{subfigure}
  \begin{subfigure}{.24\textwidth}
    \centering
    \includegraphics[width=\textwidth]{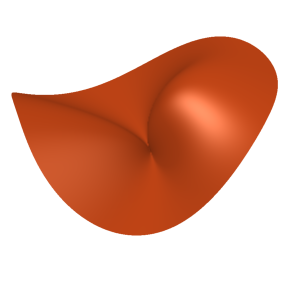}
    \caption{$D_4$}
  \end{subfigure}
  \caption{Some cubic surfaces with a $\G_m$-action that are not in
    the closure of the orbit of a generic cubic surface. We used the
    3D grapher
    \href{https://singsurf.org/parade/Cubics.php}{SURFER} for these
    images.}
  \label{fig:surfaces}
\end{figure}

Let $X \subset \P^3$ be a good cubic surface.
Let $G$ be the automorphism group of $X$ and $\G_m \to G$ a one-parameter subgroup.
The family
\[ [X/\G_m] \to B\G_m\]
yields a relation
\begin{equation}\label{eqn:isorel}
  a_{1^4}v_1^4 + a_{1^2\cdot 2} v_1^2v_2 + a_{1\cdot 3} v_1v_3 + a_{2^2}v_2^2 + a_4 v_4 = 0,
\end{equation}
where the $v_i$ are the Chern classes of the push-forward of the anti-canonical bundle.
Let us explain how to compute the push-forward of the anti-canonical bundle in this setting.
For an integer $a$, let $\chi(a)$ denote the $1$-dimensional $\G_m$ representation where the $\G_m$ action is given by
\[ t \colon v \mapsto t^a v.\]
Suppose $X \subset \P^3$ is the zero-locus of a homogeneous cubic form $F \in k[x_0,x_1,x_2,x_3]$.
Assume that the $\G_m$ acts on the variables $x_i$ by scaling, say $t \in \G_m$ acts by
\[ t \colon x_i \mapsto t^{w_i}x_i,\]
where $w_i \in \Z$.
Since $V(F)$ is $\G_m$-fixed, there exists $w \in \Z$ such that
\[ F \left(t x_0, t x_1, t x_2, t x_3\right) = t^w F\left(x_0, x_1, x_2,x_3\right).\]
Write $V^\vee$ for the $k$-vector space $\langle x_0, x_1, x_2,x_3 \rangle$.
As a $\G_m$-representation, we have
\[ V^\vee = \chi(w_1) \oplus \chi(w_2) \oplus \chi(w_3) \otimes \chi(w_4).\]
The cubic $F$ defines a $\G_m$-invariant section of $\Sym^3 V^\vee \otimes \chi(-w)$.
Thus, we can view $X$ as the zero locus of the line bundle $\O(3) \otimes \pi^*\chi(-w)$ in the projective bundle $\pi \colon [\P V/\G_m] \to B\G_m$.
By the adjunction formula, we get
\[ \omega_X^{-1} = \O_{\P V}(1) \otimes \det V \otimes \pi^*\chi(w) |_X,\]
and hence
\begin{equation}\label{eqn:rep}
\begin{split}
  \pi_*\left( \omega_X^{-1} \right) &= V^\vee \otimes \det V \otimes \chi(w) \\
  &\cong \left( \bigoplus \chi(w_i)  \right) \otimes \chi\left(-\sum w_i\right) \otimes \chi(w).
\end{split}
\end{equation}

\subsection{The family $3A_2$}
Consider
\[ X = V(x_0x_1x_3 - x_2^3).\]
A $\G_m$ stabilizing $X$ acts on the variables by weights $(a,b,0,c)$ where $a,b,c \in \Z$ are any such that $a+b+c = 0$.
Note that in this case, the weight $w$ of the cubic equation defining $X$ is $0$.
Hence, from \eqref{eqn:rep}, we get
\[ \pi_*\left(\omega^{-1}_X\right) = \chi(a) \oplus \chi(b) \oplus \chi(0) \oplus \chi(c).\]
Letting $q = c_1(\chi(1)) \in A^1(B\G_m)$, we get
\[ v_1 = 0, \quad v_2 = (ab+bc+ca) \cdot q^2, \quad v_3 = abc \cdot q^3, \quad v_4 = 0,\]
and hence
\[ v_1^4 = 0, \quad v_1^2v_2 = 0, \quad v_1v_3 = 0,\quad v_2^2 \neq 0, \quad v_4 = 0.\]
Substituting in \eqref{eqn:isorel} yields the relation
\begin{equation}\label{eqn:iso1}
  a_{2^2} = 0.
\end{equation}

\subsection{The family $A_3 + 2A_1$}
Consider
\[ X = V(x_3(x_0x_2-x_1^2) - x_0x_1^2).\]
A $\G_m$ stabilizing $X$ acts on the variables by weights $(-3,1,5,-3)$.
The weight $w$ of the cubic equations defining $X$ is $-1$.
From \eqref{eqn:rep}, we get
\[ \pi_*\left( \omega_X^{-1} \right) = \chi(-4) \oplus \chi(0) \oplus \chi(4) \oplus \chi(-4).\]
Setting $q = c_1(\chi(1)) \in A^1(B\G_m)$, we get
\[ v_1 = -4q, \quad v_2 = -16q^2, \quad v_3 = 64q^3, \quad v_4 = 0,\]
and hence
\[ v_1^4 = 256q^4 \quad v_1^2v_2 = -256q^4, \quad v_1v_3 = -256q^4, \quad v_2^2 = 256q^4, \quad v_4 = 0.\]
Substituting in \eqref{eqn:isorel} yields the relation
\begin{equation}\label{eqn:iso2}
  a_{1^4} - a_{1^2 2}- a_{13} +  a_{2^2}  = 0.
\end{equation}

\subsection{The family $A_4 + A_1$}
Consider
\[ X = V(x_3(x_0x_2-x_1^2) - x_0^2x_1).\]
A $\G_m$ stabilizing $X$ acts on the variables by weights $(1,-1,-3,3)$.
The weight $w$ of the cubic defining $X$ is $1$.
From \eqref{eqn:rep}, we get
\[ \pi_*\left( \omega_X^{-1} \right) = \chi(2) \oplus \chi(0) \oplus \chi(-2) \oplus \chi(4).\]
Setting $q = c_1(\chi(1)) \in A^1(B\G_m)$, we get
\[ v_1 = 4q, \quad v_2 = -4q^2, \quad v_3 = -16q^3, \quad v_4 = 0,\]
and hence
\[ v_1^4 = 256q^4 \quad v_1^2v_2 = -64q^4, \quad v_1v_3 = -64q^4, \quad v_2^2 = 16q^4, \quad v_4 = 0.\]
Substituting in \eqref{eqn:isorel} yields the relation
\begin{equation}\label{eqn:iso3}
  16a_{1^4} -4a_{1^2 2}- 4a_{13} + a_{2^2}  = 0.
\end{equation}

\subsection{The family $D_4$}
Consider
\[ X = V(x_3x_0^2 - x_1^3-x_2^3).\]
A $\G_m$ stabilizing $X$ acts on the variables by weights $(5,1,1,-7)$.
The weight $w$ of the cubic equation defining $X$ is $3$.
By \eqref{eqn:rep}, we get
\[ \pi_*\left( \omega_X^{-1} \right) = \chi(8) \oplus \chi(4) \oplus \chi(4) \oplus \chi(-4).\]
Setting $q = c_1(\chi(1))$, we get
\[ v_1 = 12q, \quad v_2 = 16q^2, \quad v_3 = -192q^3, \quad v_4 = -512q^4,\]
and hence
\[ v_1^4 = 20736q^4 \quad v_1^2v_2 = 2304q^4, \quad v_1v_3 = -2304q^4, \quad v_2^2 = 256q^4, \quad v_4 = -512q^4.\]
Substituting in \eqref{eqn:isorel} yields the relation
\begin{equation}\label{eqn:iso4}
  81a_{1^4} +9a_{1^2 2}- 9a_{13} +  a_{2^2}-2a_4  = 0.
\end{equation}

\subsection{Families over a base scheme}\label{sec:schematicexample}
By choosing a suitable base $B$ and a map $B \to B \G_m$, we
can construct an isotrivial family over $B$ by pulling back one of the
families above.  We describe this explicitly in an example.

Let $B = \P^4$.  Set
\[V^\vee = \mathcal O(5) \oplus \mathcal O(1) \oplus \mathcal O(1) \oplus
  \mathcal O(-7).\] Let $x_0, x_1, x_2, x_3$ be generators of the four
summands on the standard $\A^4 \subset \P^4$.  We have
the section $(x_3x_0^2-x_1^3-x_2^3)$ of
$\Sym^3 V^\vee = \mathcal O(3)^{\oplus 20}$.  Note that this section has a
pole of order $3$ along the hyperplane at infinity.  As a result, it
defines a section $\xi$ of $\Sym^3 V^\vee \otimes \mathcal O(-3)$, which is
nowhere vanishing.  Equivalently, $\xi$ is a section of
$\mathcal O_{\P V}(3) \otimes \pi^* \mathcal O(-3)$.  Our family
$\mathcal X \to B$ is defined by the zero-locus of $\xi$ in
$\P V \to B$.

\section{The equivariant orbit class}
\label{sec:tietogether}
We use the information provided by the test families to prove the theorems advertised in the introduction.
Let $V$ be a 4 dimensional vector space.
Recall our notation $W$ for the $\GL V$ representation 
\[W = \Sym^3V^\vee \otimes \det V\]
and
\[\mathscr M = \left[ W \setminus \{0\} / \GL V \right]\]
for the moduli stack of cubic surfaces.
\begin{theorem}\label{thm:eqvclass}
  There exists a non-empty open subset $U \subset W$ such that for every cubic surface $X$ represented by a point of $U$, the class of $\Orb(X)$ in $A^4_{\GL V}(W)$ is given by
  \[
    [\Orb(X)] = 1080 \cdot \left(v_{1}^{2}v_{2} - v_{1}v_{3}+ 9v_{4}\right),
  \]
  where $v_i = c_i(V)$ are the Chern classes of the standard representation of $\GL V$.
\end{theorem}
\begin{proof}
  By excision and homotopy invariance of equivariant Chow groups, we have
  \[ A^4(\mathscr M) = A^4_{\GL V}(W) = A^4_{\GL V}(\bullet).\]
  The ring $A^*_{\GL V}(\bullet)$ is generated by the classes $v_i$, and hence $A^4_{\GL V}(\bullet)$ is a free $\Z$-module generated by the monomials in $v_i$ of total degree 4.
  Therefore, for every cubic surface $X$, we have an expression
  \[[\Orb(X)] = a_{1^4}v_1^4 + a_{1^2\cdot 2} v_1^2v_2 + a_{1\cdot 3} v_1v_3 + a_{2^2}v_2^2 + a_4 v_4,\]
  for some coefficients $a_i \in \Z$.
  For a generic $X$, the families in \autoref{sec:testfamilies} give the following linear relations among the coefficients
  \begin{align*}
    16 \cdot a_{1^4} + 4 \cdot a_{1^2 2} + a_{2^2} &= 4320 \text{ from \eqref{eq:relation1}}\\
    625 a_{1^{4}} + 125 a_{1^{2}2} - 25a_{13} + 25a_{2^2} - 6 a_{4} &= 103680 \text{ from \eqref{eq:relation2}}\\
    3436a_{1^4} + 1076 a_{1^{2}2} +116 a_{13} + 316 a_{2^2} &= 1036800 \text{ from \eqref{eq:relation3}}\\
    6a_{1^{4}} + 21a_{1^{2}2}+6a_{13}+16a_{2^{2}}+a_{4} &= 25920 \text{ from \eqref{eq:relation4}}\\
    a_{2^2} &= 0  \text{ from \eqref{eqn:iso1}}\\
    a_{1^4} - a_{1^2 2}- a_{13} +  a_{2^2}  &= 0  \text{ from \eqref{eqn:iso2}}\\
    16a_{1^4} -4a_{1^2 2}- 4a_{13} +  a_{2^2}  &= 0  \text{ from \eqref{eqn:iso3}}\\
    81a_{1^4} +9a_{1^2 2}- 9a_{13} +  a_{2^2}-2a_4  &= 0  \text{ from \eqref{eqn:iso4}}.
  \end{align*}
  These equations uniquely determine all the coefficients $a_i$ to be as stated.
\end{proof}

\begin{corollary}
  The degree of the closure of the $\PGL_4$ orbit of a generic cubic surface in $\P^3$ is $96120$.
\end{corollary}
\begin{proof}
  Let $\P^3 = \P V$.
  Consider the universal cubic surface $X \subset \P V\times \P \Sym^3 V^\vee$.
  Let $\pi \colon \mathcal X \to \P \Sym^3 V^\vee$ be the second projection and $\mu \colon \P \Sym^3 V^\vee \to \mathscr M$ the map induced by $\pi$.
  Tautologically, for any cubic surface $X$, the pre-image under $\mu$ of $\Orb(X)$ is the closure of the $\PGL(V)$ orbit of $X$.
  By adjunction, we have
  \[ \omega_\pi = \O(1) \boxtimes \O(1),\]
  and hence
  \[ \mathcal V = \pi_* \omega_\pi \cong \O(1)^4.\]
  For a generic $X$, we evaluate 
  \[ [\Orb(X)] = 1080(v_1^2v_2 - v_1v_3+9v_4)\]
  in $A^4(\P \Sym^3V^\vee)$ and find that the answer is $96120$ (times the class of a codimension 4 linear subspace).
\end{proof}

\begin{corollary}
  Let $X \subset \P^4$ be a general cubic 3-fold, and let $\pi \colon \mathcal X \to \P^4$ be the family of hyperplane sections of $X$.
  Given a generic cubic surface $S$, there are $42120$ fibers of $\pi$ isomorphic to $S$.
\end{corollary}
\begin{proof}
  Note that two cubic surfaces are isomorphic if and only if they are in the same $\PGL_4$ orbit.
  Consider the pull-back of $\Orb(S)$ to the base $\P^4$ of the family.
  By the genericity of $X$ and $S$, this pull-back consists of distinct reduced points corresponding to the points of $\Orb(S) \setminus \partial \Orb(S)$.
  The statement now follows from \autoref{thm:eqvclass} and an easy Chern class computation, which we omit.
\end{proof}

\subsection{Proof of \texorpdfstring{\autoref{theorem:main}}{first theorem} from the introduction}
The theorem follows immediately from \autoref{thm:eqvclass} and \autoref{prop:goodisgood}.

\subsection{Proof of \texorpdfstring{\autoref{thm:eqvclass2}}{second theorem} from the introduction}\label{proof:eqvclass2}
Let $V$ be a 4-dimensional vector space.
Set
\[\mathscr N = \left[ \Sym^3V^\vee \setminus \{0\} / \GL V \right]. \]
Let $\mathcal V$ denote the vector bundle on $\mathscr N$ corresponding to the standard representation of $\GL V$.
The tautological non-zero section of $\Sym^3V^\vee$ on $\mathscr N$ defines a family of cubic surfaces
\[ \pi \colon \mathscr X \to \mathscr N,\]
and hence a map $\mathscr N \to \mathscr M$.
Under this map, the pull-back of an orbit closure is an orbit closure.
Note that $\mathscr X \subset \P \mathcal V$ is the zero locus of $\O(3)$.
By adjunction, we have
\[ \omega_\pi = \O(1) \otimes \det V |_{\mathscr X},\]
and hence
\[ \pi_* \omega_\pi = V^\vee \otimes \det V.\]
\autoref{thm:eqvclass} gives the class of the pull-back of a generic orbit closure in terms of the Chern classes $v_i = c_i(V^\vee \otimes \det V)$.
Writing them in terms of $c_i = c_i(V)$ yields the result.

\section{Addressing the orbit closure problem}
\label{sec:orbclosure}
We have repeatedly faced the problem of showing that a certain cubic surface is not in the orbit closure of a certain (generic) cubic surface.
In this section, we develop two sets of techniques to address this question.

\subsection{Using GIT}
The first technique uses GIT and applies broadly to any orbit closure problem.
Let $W$ be a quasi-projective variety with a linearized action of a reductive group $G$.
The basic observation is that if $x \in W$ is semi-stable and $s \in W$ is stable, then $x$ does not lie in the orbit closure of $s$.
We prove an extension of this idea that turns out to be highly effective.

\begin{proposition}\label{prop:gitproduct}
  Let $G$ be a reductive group.
  Let $W$ and $H$ be smooth quasi-projective varieties along with linearized actions of $G$.
  Let $x, s \in W$ be points satisfying the following conditions:
  \begin{enumerate}
  \item $x$ has a reductive stabilizer $T \subset G$.
  \item for some $h \in H$ fixed by the identity component of $T$, the point $(x, h) \in W \times H$ is semi-stable.
  \item for all $h \in H$, the point $(s, h) \in W \times H$ is stable.
  \end{enumerate}
  Then $x$ does not lie in the closure of the $G$-orbit of $s$.
\end{proposition}
The proof needs the following lemma.
It is well-known to experts---it appears as a remark in \cite{alp.hal.ryd:20} (see \S~1.3, Immediate consequences: 5)---but we include a proof for completeness.
Recall that $\k$ is an algebraically closed field of characteristic 0, all schemes considered are of finite type over $\k$, and a point means a $\k$-point.
\begin{lemma}\label{prop:oneparam}
  Let \(U\) be a smooth scheme over $\k$ with an action of a linear algebraic group $G$.
  Let $x \in U$ be a point whose stabilizer group $T \subset G$ is reductive.
  If $x$ lies in the closure of the $G$-orbit of a point $s \in U$, then there exists a one-parameter subgroup $\lambda \colon \G_m \to T$ and a point $s'$ in the $G$-orbit of $s$ such that
  \[ x = \lim_{t \to 0} \lambda(t) \cdot s'.\]
\end{lemma}
\begin{proof}
  By \cite[Theorem~3]{alp:10}, there exists a $T$-invariant locally closed affine $W \subset U$ containing the point $x$ such that the map
  \[ \pi \colon [W/T] \to [U / G]\]
  is affine and \'etale.
  (Note that this result is substantially easier than the main theorem of \cite{alp.hal.ryd:20}.)
  Define $Z$ as the fiber product
  \[
    \begin{tikzcd}
      Z \ar{r}\ar{d}& {[W/T]}\ar{d}{\pi}\\
      \spec \k \ar{r}{s}& {[U/G]}.
    \end{tikzcd}
  \]
  Since $\pi$ is representable, \'etale, and of finite type, $Z$ is a finite disjoint of copies of $\spec \k$.
  Let $z_1, \dots, z_n \in [W/T](\k)$ be the points of $Z$.
  Since the point $x \in [U/G](\k)$ lies in the closure of $s \in [U/G](\k)$ and the map $\pi$ is \'etale (and hence open), the point $x \in [W/T](\k)$ lies in the closure of the set $\{z_1,\dots, z_n\}$.
  But then $x$ lies in the closure of $z_i$ for some $i$.
  In other words, the point $x \in W(\k)$ lies in the closure of the $T$-orbit of a lift $s_i \in W(\k)$ of $z_i \in [W/T](\k)$.
  By the Hilbert--Mumford criterion \cite[Theorem~1.4]{kem:78}, there exists a one-parameter subgroup $\lambda \colon \G_m \to T$ such that we have the equation
  \begin{equation}\label{eqn:specialization}
    x = \lim_{t \to 0} \lambda (t) \cdot s_i.
  \end{equation}
  Since $s_i \in [W/T](\k)$ maps to $s \in [U/G](\k)$, the point $s_i \in W$ is in the same $G$-orbit as $s \in U$.
  The proof is now complete.
\end{proof}

\begin{proof}[Proof of \autoref{prop:gitproduct}]
  Let $T_0 \subset T$ be the connected  component of the identity.
  We prove the contrapositive.
  Assuming that $x$ lies in the closure of the $G$-orbit of $s$, we show that for any $T_0$-fixed $h \in H$, the point $(x, h) \in W \times H$ is unstable.

  By \autoref{prop:oneparam}, there exists a one parameter subgroup $\lambda \colon \G_m \to T_0$ and a point $s' \in W$ in the same $G$-orbit as $s$ such that
  \[ x = \lim_{t \to 0} \lambda(t) \cdot s'.\]
  For any $T_0$-fixed $h \in H$ we have
  \[
    (x, h) = \lim_{t \to 0} \lambda(t) \cdot (s', h).
  \]
  Since $(s',h)$ is $G$-stable, $(x,h)$ must be $G$-unstable.
\end{proof}

We immediately get an application to cubic surfaces.
\begin{proposition}\label{prop:isogood}
  The following singular cubic surfaces are not in the closure of the $\PGL_4$-orbit of any smooth cubic surface without an Eckardt point (the parenthesis describes the singularities):
  \begin{enumerate}
  \item $x_0x_1x_3 = x_2^3$ \quad $(3A_2)$
  \item $x_3(x_0x_2-x_1^2) = x_0x_1^2$ \quad $(A_3 + 2A_1)$
  \item $x_3(x_0x_2-x_1^2) = x_0^2x_1$ \quad $(A_4 + A_1)$
  \item $x_3x_0^2 = x_1^3 + x_2^3$ \quad $(D_4)$
  \end{enumerate}
\end{proposition}
\begin{proof}
  Let $V$ be a 4-dimensional vector space.
  Set $W = \P \Sym^3 V^\vee$ and $H = \P V^\vee$; both have natural linearized actions of $\PGL(V)$.
  The product $W \times H$ parametrizes cubic surfaces along with a hyperplane.
  We use the GIT analysis for this space carried out by Gallardo and Martinez-Garcia \cite{gal.mar:19}.
  We have a $1$-parameter choice of linearizations on $W \times V$ indexed by a positive rational number $t$.
  The $G$-linearized line bundle corresponding to $t$ is $\O(a) \boxtimes \O(b)$, where $a$ and $b$ are (sufficiently divisible) positive integers satisfying $a/b = t$.

  From \cite[Theorem~3]{sak:10}, we see that the listed cubic surfaces have a reductive stabilizer group.
  In fact, in all cases the connected component of the identity is a torus.

  Let $S$ be a smooth cubic surface without an Eckardt point corresponding to a point $s \in W$.
  Suppose we have $0 < t < 5/9$.
  From \cite[Theorem~2]{gal.mar:19}, we see  that for any $h \in H$, the point $(s,h)$ is $t$-stable.
  On the other hand, from \cite[Table 2]{gal.mar:19}, we see that for some $t$ in the range $0 < t < 5/9$, there exists a choice of $h \in H$, fixed by the connected component of the identity of the stabilizer of $s$, such that $(s,h)$ is $t$-semistable.
  We now apply \autoref{prop:gitproduct}.

  In the following table, we list the automorphism groups, the hyperplane $h$, and the value of $t$ for which $(s,h)$ is $t$-semistable.
  We denote the symmetric group on $n$ letters by $\Sigma_n$.
  The automorphism groups are from \cite{sak:10} and the GIT semi-stability assertions are from \cite{gal.mar:19}.
  \begin{center}
  \begin{tabular}{l l l l l}
    \toprule
    $X$ & Singularities of $X$& $\Aut(X)$ & $h$ & $t$\\
    \midrule
    $x_0x_1x_3 = x_2^3$& $3A_2$&$\G_m^2 \rtimes \Sigma_3$ & $x_2 = 0$ & All $t \in (0,1)$\\
    $x_3(x_0x_2-x_1^2) = x_0x_1^2$& $A_3+2A_1$&$\G_m \rtimes \Sigma_2$ & $x_2 = 0$ & $1/5$ \\
    $x_3(x_0x_2-x_1^2) = x_0^2x_1$&$A_4+A_1$& $\G_m$ & $x_2 = 0$ & $1/3$\\
    $x_3x_0^2 = x_1^3 + x_2^3$&$D_4$& $\G_m \rtimes \Sigma_3$ & $x_0 = 0$ & $3/7$\\
    \bottomrule
  \end{tabular}
\end{center}
\end{proof}

\subsection{Using 6 points in $\P^2$}
The second technique is specific to cubic surfaces.
We first reduce to the case of 6 points in $\P^2$, and then use the geometry of 1-parameter families of automorphisms of $\P^2$.
This method can show, for example, that the cubic with the $A_5$ singularity does not lie in the closure of the orbit of a generic cubic.
(This particular surface had been immune to all our previous strategies.)

\begin{proposition}\label{prop:cubictopoints}
  Let $W \subset \P^2$ be an admissible length 6 subscheme and let $X = X_W$ be the cubic surface associated to $W$.
  Let $S \subset \P^2$ be a smooth cubic surface, and suppose $X$ lies in the $\PGL_4$-orbit closure of $S$.
  Then $W$ lies in the $\PGL_3$ orbit closure of a length 6 scheme $Z \subset \P^2$ such that $\Bl_Z \P^2$ is isomorphic to $S$.
\end{proposition}
\begin{proof}
  Since $W$ is in the orbit closure of $S$, there exists a $\k$-algebra DVR $R$ with fraction field $K$ and residue field $k$ with a family of cubic surfaces $\pi \colon \mathcal X \to \spec R$ whose central fiber $\mathcal X_0$ is isomorphic to $X$ and whose general fiber $\mathcal X_K$ is isomorphic to $S_K = S \times_k K$.
  We know that $X$ has only ADE singularities (\autoref{prop:admissibleblowup}) and hence we can resolve the singularities of $X$ in the family.
  That is, there exists a smooth and projective $\pi' \colon \mathcal X' \to \spec R$ and $\beta \colon \mathcal X' \to \mathcal X$ such that $\beta_K$ is an isomorphism and $\beta_0$ is the minimal resolution of singularities; see \cite{bri:70} or \cite{tju:70}.

  Set $X' = \mathcal X'_0$.
  We know from the proof of \autoref{prop:good} that $X'$ is also the minimal desingularization of $\Bl_Z\P^2$.
  Let $X' \to \P^2$ be the composite $X' \to \Bl_Z \P^2 \to \P^2$, and let $L'$ be pullback of $\O(1)$ to $\mathcal X'_0$.
  It is easy to see that $h^i(X', L') = 0$ for $i > 0$ and $h^0(X', L') = 3$.
  Since $H^2(X', \O_{X'}) = 0$, there are no obstructions to deforming $L'$---it extends to compatible family of line bundles on $\mathcal X' \times_R \spec R/t^n$ for every $n$, where $t \in R$ is the uniformizer. 
  By Artin's approximation theorem \cite{art:69} applied to the relative Picard functor of $\pi'$, we deduce that $L'$ extends to a line bundle $\mathcal L$ on $\mathcal X'$, after possibly replacing $R$ by a finite \'etale cover.
  By cohomology and base change, we see that $\pi'_* \mathcal L$ is a free $R$-module of rank 3.
  It is also easy to see that the evaluation map $\pi'^*\pi'_* \mathcal L \to \mathcal L$ is surjective.
  After choosing an isomorphism $\pi'_* \mathcal L \cong R^3$, we get a map $\alpha \colon \mathcal X' \to \P^2_R$, which restricts to the original map $X' \to \P^2$ on the central fiber.
  
  Recall that $\mathcal X'_K = \mathcal X_K = S \times_k K$.
  Since $H^1(S, \O_S) = 0$, it follows that the line bundle $\mathcal L$ on $S \times_k K$ is constant, that is, isomorphic to $L \otimes_k K$ for some line bundle $L$ on $S$.
  The complete linear system $|L|$ gives a map $a \colon S \to \P^2$.
  It is easy to check that the map $\alpha \colon \mathcal X'_K \to \P^2_K$ and $a_K \colon \mathcal X'_K \to \P^2_K$ are equal, up to an automorphism of $\P^2_K$.
  
  The natural map $\alpha^* \omega_{\P^2_R/R} \to \omega_{\mathcal X'/R}$ gives a map $\omega^{-1}_{\mathcal X'/R} \to \alpha^* \omega^{-1}_{\P^2_R/R}$ and hence an evaluation map
  \[ \pi'_* \left(\omega_{\mathcal X'/R}\right) \otimes_R \O_{\P^2_R} \to \omega_{\P^2_R/R}^{-1}.\]
  Tensor by $\omega_{\P^2_R/R}$ to get
  \[ \pi'_* \left(\omega_{\mathcal X'/R}\right) \otimes_R \omega_{\P^2_R/R} \to \O_{\P^2_R}.\]
  Then the cokernel is $\O_{\mathcal Z}$ for a closed subscheme $\mathcal Z \subset \P^2_R$.
  Note that both the general and the special fiber of $\mathcal Z \to \spec R$ have length 6, and hence $\mathcal Z \to \spec R$ is flat.
  The central fiber is the subscheme $W \subset \P^2$ and, up to an automorphism of $\P^2_K$, the general fiber is $Z \times_k K \subset \P^2_K$, where $Z \subset \P^2$ is such that $S \cong \Bl_Z \P^2$.
  The proof is now complete.
\end{proof}

\subsubsection{The anatomy of linear rational maps $\P_t^2 \dashrightarrow \P_t^2$}\label{sec:ratmap}
To study orbit closures of collections of points in $\P^2$, we must understand one parameter families of automorphisms of $\P^2$.

Let $R$ be a $\k$-algebra DVR with residue field $\k$, fraction field $K$, and uniformizer $t$.
For an $R$-module $M$, we set $M_K = M \otimes_R K$ and $M_0 = M \otimes_R \k$.
Let $\phi \colon \P^2_K \to \P^2_K$ be an isomorphism.
We analyze the corresponding rational map $\P^2_R \dashrightarrow \P^2_R$.
The analysis should hold for $\P^n$ for any $n$; we stick to $n = 2$ to keep us focused.

Let $M$ and $N$ be free $R$-modules of rank 3, and suppose $\phi \colon M_K \to N_K$ is a $K$-linear isomorphism.
By scaling $\phi$ by the correct power of the uniformizer, assume that $\phi$ induces a map $\phi \colon M \to N$ such that $\phi_0 \colon M_0 \to N_0$ is non-zero.
By the \emph{rank} of $\phi$, we mean the rank of $\phi_0$.

By making a change of basis on $M$ and $N$, we may assume that $\phi$ is given by a diagonal matrix
\[
  \phi \equiv
  \begin{pmatrix}
    1 & & \\
     & t^a & \\
    & & t^b
  \end{pmatrix},
\]
where $0 \leq a \leq b$ are integers.

If $a = b = 0$, then there is nothing to analyze.

Suppose $a = 0$ and $b > 0$, or equivalently, $\rk \phi_0 = 2$.
In this case, the rational map $\phi$ is defined away from the point $[0:0:1]$ in $\P M_0$.
It sends the point $[a:b:c]$ to the point $[a:b:0]$, which lies on the line $Z = 0$.
Said in a coordinate free manner, the rational map $\phi$ is defined away from a particular point $p \in \P M_0$.
For $q \in \P M_0 \setminus \{0\}$, the image $\phi(q)$ lies on a particular line $L \subset \P N_0$, and the map
\begin{equation}\label{eqn:rk2}
  \phi \colon \P M_0 \setminus \{p\} \to L
\end{equation}
is the linear projection from $p$ followed by an isomorphism.

Suppose $a > 0$, or equivalently $\rk \phi_0 = 1$.
In this case, the rational map $\phi$ is defined away from the line $X = 0$ in $\P M_0$.
It sends the point $[a:b:c]$ to the point $[1:0:0]$.
Said in a coordinate free manner, the rational map $\phi$ is defined away from a particular line $L \subset \P M_0$ and the map
\begin{equation}\label{eqn:rk1}
  \phi \colon \P M_0 \setminus L \to \P N_0
\end{equation}
is constant.

Using the matrix of $\phi$, we can explicitly construct a resolution of the rational map $\phi \colon \P^2_R \dashrightarrow \P^2_R$ using a sequence of elementary transformations.
We get the description of the maps \eqref{eqn:rk2} and \eqref{eqn:rk1} also from this resolution.
Let $L \subset \P^2$ be the locus of indeterminacy of $\phi$ (with the reduced scheme structure).
Let $\widetilde P \to \P^2_R$ be the blow up of $L$.
Then the central fiber of $\widetilde P$ is the union of $\Bl_L\P^2$ and the exceptional divisor $E$.
If $\phi$ has rank 1, then $L$ is a line, and these two components are $\P^2$ and the Hirzebruch surface $\F_1$, respectively, meeting along $L \subset \P^2$ and the unique $(-1)$ curve on $\F_1$.
If $\phi$ has rank 2, then $L$ is a point, and these two components are $\F_1$ and $\P^2$, respectively, meeting along the $(-1)$-curve on $\F_1$ and a line in $\P^2$.
In both cases, the first component can be contracted: to a point in the rank 1 case and to a line in the rank 2 case.
The resulting threefold is again isomorphic to $\P^2_R$.
The rational map $\P^2_R \dashrightarrow \P^2_R$ after this operation corresponds to the diagonal matrix with entries $(1,t^{a-1}, t^{b-1})$ in the rank 1 case and $(1, 1, t^{b-1})$ in the rank 2 case.
Observe that the new locus of indeterminacy is disjoint from the image of the contracted component.
We repeat the procedure until we reach the identity matrix.
Alternatively, we can do all the blow-ups first, and then all the contractions.
We then get a diagram
\[
\begin{tikzpicture}[node distance=5em]
  \node (t) {$\widehat P$};
  \node (l) [below left of=t] {$\P^2_R$};
  \node (r) [below right of=t] {$\P^2_R$};
  \path [->] (l) edge[dashed]  node[above] {$\phi$}  (r) (t) edge (l) (t) edge (r);
\end{tikzpicture}
\]
The central fiber of $\widehat P$ has an accordion like structure $P_0 \cup \dots \cup P_n$ (see \autoref{fig:accordion}), where each $P_i$ is a copy of $\F_1$ for $i = 1, \dots, n-1$, meeting the previous surface along the $(-1)$-curve and the next one along a $(+1)$-curve.
The first surface $P_0$ is $\F_1$ if $a = 0$ (that is, $\phi$ has rank 2) or $\P^2$ if $a > 0$ (that is, $\phi$ has rank 1).
Likewise, the last surface $P_n$ is $\F_1$ if $a = b$ or $\P^2$ if $b > a$.

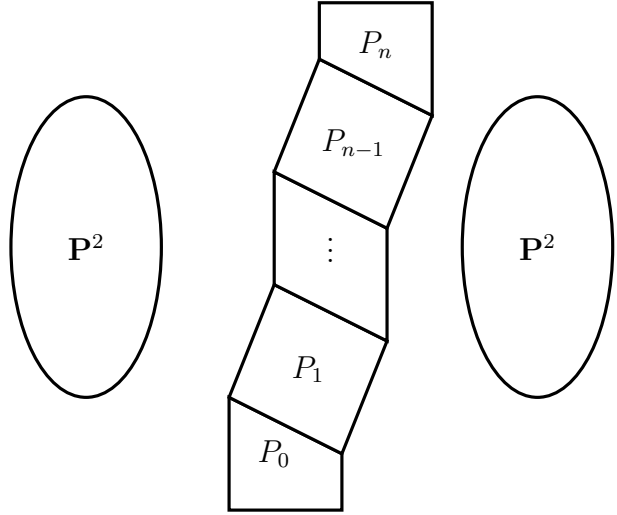
\begin{figure}
  \centering
  \begin{tikzpicture}[very thick]
  \begin{scope}[xshift=-3cm]
    \draw (0,0) ellipse (1 and 2);
    \draw (0,0) node {$\P^2$};
  \end{scope}
  \begin{scope}[yshift=-2cm, xshift=-0.5cm, yscale=0.75, xscale=1.5]
    \draw[fill=white] (-0.4,-2.0) -- (0.6,-2.0) -- (0.6, -1) -- (-0.4,0) -- cycle;
    \draw[fill=white] (0.6, -1) -- (-0.4,0) -- (0, 2) -- (1,1) -- cycle;
    \draw[fill=white] (0, 2) -- (1,1) -- (1,3) -- (0,4) -- cycle;
    \draw[fill=white] (1, 3) -- (0,4) -- (0.4, 6) -- (1.4,5) -- cycle;
    \draw[fill=white] (0.4, 6) -- (1.4,5) -- (1.4,7) -- (0.4,7) -- cycle;
    \draw
    (0,-1.0) node {$P_0$}
    (0.3,0.5) node {$P_1$}
    (0.5, 2.75) node {$\vdots$}
    (0.7, 4.5) node {$P_{n-1}$}
    (0.9, 6.25) node {$P_{n}$};
    \draw (0,-2) node (X) {};
  \end{scope}
  \begin{scope}[xshift=+3cm]
    \draw (0,0) ellipse (1 and 2);
    \draw (0,0) node {$\P^2$};
  \end{scope}
\end{tikzpicture}
  \caption{A family of surfaces with an accordion like central fiber arises while resolving a linear rational map $\P^2_R \dashrightarrow \P^2_R$.}
  \label{fig:accordion}
\end{figure}

\subsubsection{Orbit closure of 6 general points in $\P^2$}\label{sec:orb6}
We use the analysis in \autoref{sec:ratmap} to understand limits of a general set of six points in $\P^2$ under a one-parameter family of automorphisms.
Let $Z \subset \P^2$ be a configuration of 6 distinct points, with no three collinear.
A point $p \in \P^2 \setminus Z$ is a \emph{star point} if, for some ordering $\{z_1, \dots, z_6\}$ of $Z$, the lines $\langle z_1, z_2 \rangle$ and $\langle  z_3, z_4 \rangle$ and $\langle  z_5, z_6 \rangle$ are concurrent at $p$.
A \emph{star-free} configuration is one without any star points.

\begin{proposition}\label{prop:limit4}
  Let $Z \subset \P^2$ be a set of 6 points forming a star-free configuration.
  Suppose $W \subset \P^2$ is in the $\PGL_3$ orbit closure of $Z$.
  Then there exists a line in $\P^2$ containing at least 4 distinct points of $W$ or there exists a point in $\P^2$ at which $W$ has multiplicity at least 4.
\end{proposition}
\begin{proof}
  There exists a $\k$-algebra DVR $R$ with fraction field $K$ and $\phi \in \Aut(\P^2_K)$ such that $W$ is the flat limit of $\phi(Z_K)$.
  Consider the rational map
  \[ \phi \colon \P^2_R \dashrightarrow \P^2_R.\]
  We use $\phi$ to understand $W$ as a cycle.
  For $z \in Z$, let $z_K$ denote the constant section $z \times_k K$ of $\P^2_K$, and let $\phi(z)_0$ be the flat limit of $\phi(z_K)$.
  Then the cycle $[W]$ of $W$ is simply
  \[ [W] = \sum_{z \in Z} \phi(z)_0.\]

  Suppose $\phi$ has rank 2.
  Then the indeterminacy locus of $\phi$ consists of a single point $p$ in the central fiber, and the map $\phi_0 \colon \P^2 \dashrightarrow \P^2$
  is the linear projection with center $p$ onto a line $L \subset \P^2$ (see \eqref{eqn:rk2}).
  If $p \not \in Z$, then  $\phi_0$ is defined for every $z \in Z$, and $\phi(z)_0$ is simply $\phi_0(z)$.
  Since $Z$ is star-free, the linear projection from $p$ can identify at most 2 pairs of points from $Z$.
  As a result, the set $\{\phi_0(z) \mid z \in Z\}$ consists of at least 4 distinct points on $L$.
  If $p \in Z$, consider $Z' = Z \setminus \{p\}$.
  Then $\phi_0$ is defined on $Z'$, and $\phi(z)_0$ is simply $\phi_0(z)$ for $z \in Z'$.
  Since no 3 points of $Z$ are collinear, the linear projection from $p$ maps the points of $Z'$ to distinct points.
  As a result, the set $\{\phi_0(z) \mid z \in Z\}$ actually consists of 5 distinct points on $L$.

  Suppose $\phi$ has rank 1.
  Then the indeterminacy locus of $\phi_0$ is a line $L \subset \P^2$ and $\phi_0$ is constant on $\P^2 \setminus L$.
  Let $Z' = Z \setminus L$.
  Since no three points of $Z$ are collinear, $L$ contains at most 2 points of $Z$, and hence $Z'$ contains at least 4 points.
  By construction, $\phi_0$ is defined on $Z'$, and hence $\phi(z)_0 = \phi_0(z)$ for $z \in Z'$.
  Since $\phi_0$ is constant, we conclude that $W$ contains a point of multiplicity at least 4.
\end{proof}

\begin{proposition}\label{prop:pq4r}
  Let $W \subset \P^2$ be a curvilinear scheme of length 6 whose cycle is $p + q + 4r$, where $p,q,r$ are distinct and $T_rW$ does not contain $p$ or $q$.
  Then $W$ does not lie in the $\PGL_3$ orbit closure of a star-free configuration $Z$.
\end{proposition}
\begin{proof}
  We prove the contrapositive.
  Let $\phi \in \Aut(\P^2_K)$ be an automorphism such that $W$ is the flat limit of $\phi(Z_K)$.
  We show that $T_r W$ contains $p$ or $q$.
  The key idea is to consider flat limits of lines in $\P^2$ under the action of $\phi$.
  To do so, we use the same analysis as before but for the map between dual projective spaces
  \[\phi^\vee \colon (\P^2)_R^\vee \dashrightarrow (\P^2)_R^\vee.\]
  If $A$ is the $3 \times 3$ matrix defining $\phi$, then the matrix defining $\phi^\vee$ is $(A^T)^{-1}$.

  If $\phi_0^\vee$ is defined for a line $\ell \in (\P^2)^\vee$, then the flat limit of the family of lines $\phi(\ell_K) \subset \P^2_K$ is the line corresponding to $\phi_0^\vee(\ell)$.
  The statements \eqref{eqn:rk2} and \eqref{eqn:rk1} for the dual projective space say the following.
  If $\phi^\vee$ has rank 1, then the indeterminacy locus of $\phi^\vee_0$ consists of a particular $\Lambda \in (\P^2)^\vee$.
  For $\ell \neq \Lambda$, the map $\phi_0^\vee$ is the composite of $\ell \mapsto \ell \cap \Lambda$ and a linear inclusion $\Lambda \to (\P^2)^\vee$.
  In this case, a fiber of $\phi_0^\vee$ consists of a collection of concurrent lines, with the point of concurrency on $\Lambda$.
  If $\phi^\vee$ has rank 1, then the indeterminacy locus of $\phi^\vee_0$ consists of lines through a particular point $p \in \P^2$.
  For $\ell \in (\P^2)^\vee$ not containing $p$, the map $\phi_0^\vee$ is constant.

  Consider a line $\ell$ joining pairs of points of $Z$.
  By semi-continuity, the flat limit of $\phi(\ell_K)$ must intersect $W$ in a scheme of length at least 2.
  There are at most four lines in $\P^2$ that intersect $W$ in a scheme of length at least 2: the lines $\overline{pq}$, $\overline{pr}$, $\overline{qr}$, and $T_rW$.
  In particular, if $\phi_0^\vee$ is defined at $\ell$, then $\phi_0^\vee(\ell)$ must be one of these 4.
  This immediately implies that $\phi^\vee$ cannot have rank 1.
  To see this, suppose it has rank 1.
  Consider the ${6 \choose  2} = 15$ lines joining pairs of points of $Z$.
  If $\phi_0^\vee$ takes only 4 values on these lines (or may be undefined on one of them), we would be able to partition them into 4 collections of concurrent lines, such that the points of concurrency are collinear.
  But it is easy to check that this is impossible for a star-free configuration $Z$.

  We have now concluded that $\phi^\vee$ has rank 2, and hence $\phi_0^\vee$ is constant on its domain.
  Suppose the constant value of $\phi_0^\vee$ corresponds to the line $L \subset \P^2$.
  Suppose $z_1 \in Z$ (resp. $z_2 \in Z$) is such that the flat limit of $\phi({z_1}_K)$ is $p$ (resp. $q$).
  Consider the 8 lines $\overline {z_i z}$ for $i= 1,2$ and $z \in Z \setminus \{z_1,z_2\}$.
  Since these lines are not all concurrent, $\phi_0^\vee$ is defined for at least one of them, say $\ell = \overline {z_1 z}$.
  Then $\phi^\vee_0(\ell)$ corresponds to the line $L$.
  Since $\ell$ contains $z_1$, the flat limit $L$ of $\phi(\ell_K)$ contains the flat limit $p$ of $\phi({z_1}_K)$.

  Now consider the ${4 \choose 2} = 6$ lines joining pairs of points of $Z \setminus \{z_1,z_2\}$.
  Since the 6 lines are not all concurrent, $\phi_0^\vee$ is defined on at least one of them, say $\ell$.
  Then the flat limit of $\phi(\ell_K)$ is also $L$.
  Since the flat limit of a point of $Z \setminus \{z_1,z_2\}$ is $r$, by semi-continuity, we have
  \[ \mult_r(L \cap W) \geq 2.\]
  In other words, $L = T_rW$.
  But then $T_rW$ contains $p$.
\end{proof}
\begin{remark}\label{prop:2p4r}
  The same conclusion holds for a curvilinear $W$ with cycle structure $2p + 4r$ such that $T_rW$ does not contain $p$.
  The same proof works with $p = q$.
\end{remark}

\subsubsection{Applications}
Let $X \subset \P^{3}$ denote the cubic surface defined in homogeneous coordinates $\left(x_{0}, x_{1}, x_{2}, x_{3}\right)$ by the equation \[x_{3}f_{2} - f_{3} = 0,\] where $f_{2} = x_{0}x_{1}$ and $f_{3} = x_{0}^3 + x_{1}^{3}-x_{1}x_{2}^{2}$.
This is the unique cubic surface possessing a single $A_5$ singularity, up to projective equivalence.
For the reader's amusement, we include a real picture of $X$ in \autoref{figure:cubicA5}.
\begin{figure}
  \centering
  \includegraphics[width=0.25\textwidth]{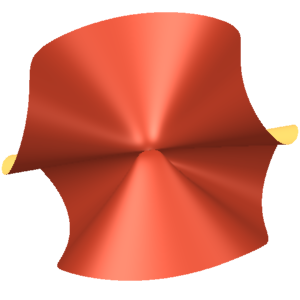}
  \caption{The unique cubic surface $X$ with an $A_5$ singularity}
  \label{figure:cubicA5}
\end{figure}

The surface $X$ is the cubic surface associated to the length 6 subscheme $W \subset \P^2$ defined by $f_2 = 0$ and $f_3 = 0$.
The scheme $W$ has the cycle $p + q + 4r$, where $p$, $q$, $r$ are distinct but collinear, and the tangent line $T_rW$ is distinct from the line $\overline{pqr}$.
The identity component of the automorphism group of $X$ is $\G_a$.
\begin{proposition}
  Let $S$ be a smooth cubic surface without an Eckardt point.
  Then $X$ is not in the $\PGL_4$ orbit closure of $S$.
\end{proposition}
\begin{proof}
  We know that $S$ is isomorphic to $\Bl_Z \P^2$, where $Z \subset \P^2$ is a star-free configuration.
  The result now follows from \autoref{prop:cubictopoints} and \autoref{prop:pq4r}.
\end{proof}

\bibliographystyle{plain}
\bibliography{references}

\begin{thebibliography}{10}

\bibitem{alp:10}
Jarod Alper.
\newblock On the local quotient structure of {A}rtin stacks.
\newblock {\em J. Pure Appl. Algebra}, 214(9):1576--1591, 2010.

\bibitem{alp.hal.ryd:20}
Jarod Alper, Jack Hall, and David Rydh.
\newblock A {L}una \'{e}tale slice theorem for algebraic stacks.
\newblock {\em Ann. of Math. (2)}, 191(3):675--738, 2020.

\bibitem{alu.fab:93}
Paolo Aluffi and Carel Faber.
\newblock Linear orbits of smooth plane curves.
\newblock {\em J. Algebraic Geom.}, 2(1):155--184, 1993.

\bibitem{alu.fab:00}
Paolo Aluffi and Carel Faber.
\newblock Linear orbits of arbitrary plane curves.
\newblock volume~48, pages 1--37. 2000.
\newblock Dedicated to William Fulton on the occasion of his 60th birthday.

\bibitem{Aluffi1993}
Paolo Aluffi and Carol Faber.
\newblock Linear orbits of d-tuples of points in $\mathbf{P}^1$.
\newblock {\em Journal für die reine und angewandte Mathematik}, 445:205--220,
  1993.

\bibitem{art:69}
M.~{Artin}.
\newblock {Algebraic approximation of structures over complete local rings}.
\newblock {\em {Publ. Math., Inst. Hautes \'Etud. Sci.}}, 36:23--58, 1969.

\bibitem{bek:82}
N.~D. {Beklemishev}.
\newblock {Invariants of cubic forms in four variables}.
\newblock {\em {Mosc. Univ. Math. Bull.}}, 37(2):54--62, 1982.

\bibitem{bri:70}
Egbert Brieskorn.
\newblock Singular elements of semi-simple algebraic groups.
\newblock In {\em Actes du Congres International des Math{\'e}maticiens (Nice,
  1970)}, volume~2, pages 279--284, 1970.

\bibitem{bru.wal:79}
J.~W. {Bruce} and C.~T.~C. {Wall}.
\newblock {On the classification of cubic surfaces}.
\newblock {\em {J. Lond. Math. Soc., II. Ser.}}, 19:245--256, 1979.

\bibitem{bru-i-mon.tim.wei:20}
Laura Brustenga I~Moncus\'{\i}, Sascha Timme, and Madeleine Weinstein.
\newblock 96120: the degree of the linear orbit of a cubic surface.
\newblock {\em Matematiche (Catania)}, 75(2):425--437, 2020.

\bibitem{cay:49}
Arthur Cayley.
\newblock On the triple tangent planes of surfaces of third order.
\newblock {\em The Cambridge and Dublin Mathematical Journal}, 4:118--132,
  1849.

\bibitem{caz.ska:20}
Elisa Cazzador and Bj\"orn Skauli.
\newblock Towards the degree of the {${\rm PGL}(4)$}-orbit of a cubic surface.
\newblock {\em Matematiche (Catania)}, 75(2):439--456, 2020.

\bibitem{cle:61*1}
A.~Clebsch.
\newblock Ueber die {K}notenpunkte der {H}esseschen {F}l\"{a}che, insbesondere
  bei {O}berfl\"{a}chen dritter {O}rdnung.
\newblock {\em J. Reine Angew. Math.}, 59:193--228, 1861.

\bibitem{cle:61}
A.~Clebsch.
\newblock Ueber eine {T}ransformation der homogenen {F}unctionen dritter
  {O}rdnung mit vier {V}er\"{a}nderlichen.
\newblock {\em J. Reine Angew. Math.}, 58:109--126, 1861.

\bibitem{edi.gra:98}
Dan Edidin and William Graham.
\newblock Equivariant intersection theory.
\newblock {\em Invent. Math.}, 131(3):595--634, 1998.

\bibitem{enr.fan:97}
F~Enriques and G~Fano.
\newblock Sui gruppi di trasformazioni cremoniane dello spazio.
\newblock {\em Annali di Matematica pura id applicata}, 15(2):59--98, 1897.

\bibitem{ful:78}
William Fulton.
\newblock A note on residual intersections and the double point formula.
\newblock {\em Acta Mathematica}, 140(1):93--101, 1978.

\bibitem{gal.mar:19}
Patricio Gallardo and Jesus Martinez-Garcia.
\newblock Moduli of cubic surfaces and their anticanonical divisors.
\newblock {\em Rev. Mat. Complut.}, 32(3):853--873, 2019.

\bibitem{has:03}
Brendan {Hassett}.
\newblock {Moduli spaces of weighted pointed stable curves.}
\newblock {\em {Adv. Math.}}, 173(2):316--352, 2003.

\bibitem{kem:78}
George~R. Kempf.
\newblock Instability in invariant theory.
\newblock {\em Ann. of Math. (2)}, 108(2):299--316, 1978.

\bibitem{mir.des:89}
JM~Miret and S~Xamb{\'o} Descamps.
\newblock Geometry of complete cuspidal plane cubics.
\newblock In {\em Algebraic Curves and Projective Geometry}, pages 195--234.
  Springer, 1989.

\bibitem{mum.fog.kir:94}
D.~Mumford, J.~Fogarty, and F.~Kirwan.
\newblock {\em Geometric invariant theory}, volume~34 of {\em Ergebnisse der
  Mathematik und ihrer Grenzgebiete (2) [Results in Mathematics and Related
  Areas (2)]}.
\newblock Springer-Verlag, Berlin, third edition, 1994.

\bibitem{mum:77}
David Mumford.
\newblock Stability of projective varieties.
\newblock {\em Enseign. Math. (2)}, 23(1-2):39--110, 1977.

\bibitem{ran.stu:20}
Kristian Ranestad and Bernd Sturmfels.
\newblock Twenty-seven questions about the cubic surface.
\newblock {\em Matematiche (Catania)}, 75(2):411--424, 2020.

\bibitem{sak:10}
Yoshiyuki Sakamaki.
\newblock Automorphism groups on normal singular cubic surfaces with no
  parameters.
\newblock {\em Trans. Amer. Math. Soc.}, 362(5):2641--2666, 2010.

\bibitem{sal:49}
George Salmon.
\newblock On the triple tangent planes to a surface of the third order.
\newblock {\em The Cambridge and Dublin Mathematical Journal}, 4:252--260,
  1849.

\bibitem{sal:60}
George Salmon.
\newblock Xiv. on quaternary cubics.
\newblock {\em Philosophical Transactions of the Royal Society of London},
  150:229--239, 1860.

\bibitem{sch:58}
Ludwig Schl{\"a}fli.
\newblock An attempt to determine the twenty-seven lines upon a surface of the
  third order, and to derive such surfaces in species, in reference to the
  reality of the lines upon the surface.
\newblock {\em Quarterly Journal of Pure and Applied Mathematics}, 2:110--120,
  1858.

\bibitem{tju:70}
G.~N. Tjurina.
\newblock Resolution of singularities of flat deformations of double rational
  points.
\newblock {\em Funkcional. Anal. i Prilo\v{z}en.}, 4(1):77--83, 1970.

\end{thebibliography}
\end{document}